\documentclass[11pt]{article}
\usepackage{amsmath}
\usepackage{amsthm}
\usepackage{amssymb}
\usepackage{url}
\usepackage{mathrsfs}
\usepackage{amscd}
\usepackage{multirow}
\usepackage{bigdelim}
\usepackage{graphicx}
\usepackage{picture}
\usepackage{subfigure}
\usepackage{enumerate}

\usepackage{xcolor}
\usepackage[normalem]{ulem}
\newcommand{\green}{\color{green!55!black}}
\newcommand{\COMMENT}[1]{{\green\em #1}}

\usepackage{cancel}

\newcommand{\tzeta}{\tilde{\zeta}}
\newcommand{\tz}{\tilde{z}}

\newcommand{\tc}{\tilde{\chi}}
\newcommand{\tphi}{\tilde{\phi}}

\newcommand{\RR}{\mathbb{R}}

\newcommand{\tk}{\tilde{k}}
\newcommand{\tK}{\tilde{K}}

\newcommand{\tG}{\tilde{\Gamma}}
\newcommand{\tL}{\tilde{\mathcal{K}}}
\newcommand{\K}{\mathcal{K}}
\newcommand{\tl}{\tilde{\lambda}}

\newcommand{\Po}{\mathcal{P}}
\newcommand{\p}{p}
\newcommand{\tp}{\tilde{\p}}
\newcommand{\stro}{\mathfrak{s}}

\newcommand{\T}{\mathbb{T}}
\newcommand{\Tt}{\mathbb{T}_T}
\newcommand{\X}{\mathcal{X}}
\newcommand{\U}{\mathcal{U}}
\newcommand{\tW}{\tilde{W}}
\newcommand{\W}{\mathcal{W}}
\newcommand{\Ered}{E_{\text{red}}}
\newcommand{\condsHomo}{\eqref{eq:restriction1}--\eqref{eq:restriction3}}
\newcommand{\condsh}{ {\em h1}--{\em h2}}

\newtheorem{lemma}{Lemma}
\newtheorem{remark}{Remark}
\newtheorem{proposition}{Proposition}
\graphicspath{
{figures/}
{figures/inner_dynamics/}
{figures/global_dynamics/}
{figures/Kf_theta-c-w/}
{figures/Kf_theta-c-w2/}
{figures/Kf_theta-c-w3/}
{figures/Kf_theta-c-w4/}
{figures/Kf_x-y-w/}
{figures/Kf_x-y-w2/}
{figures/coupled_piezo/}
{figures/figure_eight/}
{figures/unpert_periods/}
{figures/only_damping/}
}

\begin{document}
\unitlength=\textwidth
\title{Invariant manifolds and the parameterization method in coupled
energy harvesting piezoelectric oscillators\thanks{The research
leading to these results has received funding from the People
Programme (Marie Curie Actions) of the European Union's Seventh
Framework Programme (FP7/2007-2013) under REA grant agreement no.
609405 (COFUNDPostdocDTU).  This work has also been partially
supported by  MINECO MTM2015-65715-P Spanish grant. We acknowledge the
use of the UPC Dynamical Systems group's cluster for research
computing ({https://dynamicalsystems.upc.edu/en/computing/})}}
\author{Albert Granados\footnotemark[2]}
\date{}
\maketitle
\renewcommand{\thefootnote}{\fnsymbol{footnote}}
\footnotetext[2]{algr@dtu,dk, Department of Applied Mathematics and
Computer Science, Technical University of Denmark, Building 303B, 2800
Kgns. Lyngby, Denmark.}
\renewcommand{\thefootnote}{\arabic{footnote}}
\begin{abstract}
Energy harvesting systems based on oscillators aim to capture energy
from mechanical oscillations and convert it into electrical energy.
Widely extended are those based on piezoelectric materials, whose
dynamics are Hamiltonian submitted to different sources of
dissipation: damping and coupling. These dissipations bring the system
to low energy regimes, which is not desired in long term as it
diminishes the absorbed energy.  To avoid or to minimize such
situations, we propose that the coupling of two oscillators could
benefit from theory of Arnold diffusion. Such phenomenon studies
$O(1)$ energy variations in Hamiltonian systems and hence could be
very useful in energy harvesting applications.
This article is a first step towards this goal. We consider two
piezoelectric beams submitted to a small forcing and coupled through
an electric circuit. By considering the coupling, damping and forcing
as perturbations, we prove that the unperturbed system possesses a
$4$-dimensional Normally Hyperbolic Invariant Manifold with $5$ and
$4$-dimensional stable and unstable manifolds, respectively. These are
locally unique after the perturbation. By means of the
parameterization method, we numerically compute parameterizations of
the perturbed manifold, its stable and unstable manifolds and  study
its inner dynamics. We show evidence of homoclinic connections when
the perturbation is switched on.
\end{abstract}
\tableofcontents
\iffalse
TODO:
\begin{itemize}
\item try forcing with higher harmonics
\end{itemize}
\fi
\section{Introduction}
Energy harvesting systems consists of devices able to absorb energy
from the environment and, typically, electrically power a load or accumulate
electrical energy in accumulators (super capacitors or batteries) for
later use. One of the most extended approaches is by means of
piezoelectric materials, which, under a mechanical strain, generate an
electric charge. Such materials are however mostly observed working in the inverse
way in, for example, most cell phones: they generate a
vibration when driven by a varying voltage.\\
Most energy harvesting systems based on piezoelectric materials aim to
absorb energy from machine vibrations, pedestrian walks or wind
turbulences, and can power loads ranging from tiny sensors through
small vibrations to small communities through networks of large
piezoelectric ``towers'' submitted to wind turbulences. One of the
most extended configurations consists of a piezoelectric beam or
cantilever. Due to the viscous nature of the piezoelectric materials,
they behave like damped oscillators which, in absence of a strong
enough external forcing, tend to oscillate with small amplitude close
to the resting position. In order to benefit higher energy
oscillations, a typical approach consists of locating two magnets in
inverse position as in Figure~\ref{fig:single_piezo}. If the magnets
are strong enough with respect to the damping of the beam, in the
absence of an external forcing,  the resting vertical position
(previously an attracting focus) becomes an unstable (saddle)
equilibrium and two new attracting foci appear pointing to each of the
magnets.\\
\begin{figure}
\begin{center}
\begin{picture}(1,0.4)
\put(0,0){
\subfigure[\label{fig:damped_beam}]{\includegraphics[width=0.4\textwidth]{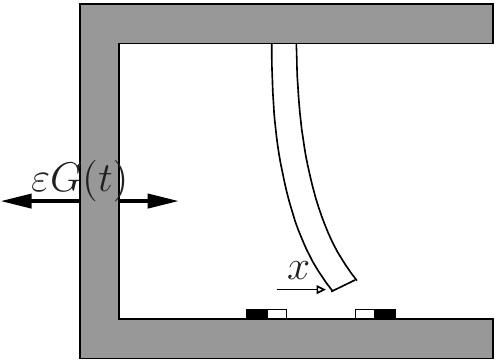}}
}
\put(0.51,0){
\subfigure[\label{fig:single_piezo}]{\includegraphics[width=0.4\textwidth]{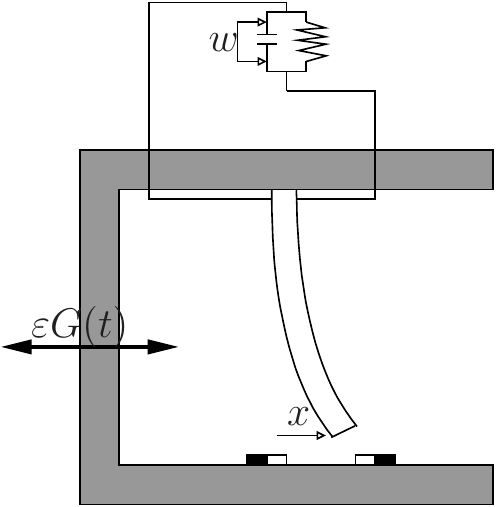}}
}
\end{picture}
\end{center}
\caption{Generic elastic beam (a) and piezoelectric beam (b) subject
to the influence of two magnets and a small periodic forcing.}
\end{figure}
The equations of motions for a generic (not necessarily piezoelectric)
damped and forced beam with magnets as the one shown in
Figure~\ref{fig:damped_beam} were first derived in~\cite{MooHol79},
which were shown to be a Duffing equation:
\begin{equation*}
\ddot{x}+2\zeta \dot{x}-\frac{1}{2}x\left( 1-x^2 \right)=\varepsilon
G(t),
\end{equation*}
where $x$ is the dimensionless horizontal displacement of the lower end, $\zeta$
is the damping coefficient and $\varepsilon G(t)$ a small periodic
forcing. When a piezelectric beam is connected to a load in the upper
end (as in Figure~\ref{fig:single_piezo}), the load receives a certain
power, a voltage $w$, whose time-derivative is proportional to the
speed of lower displacement. From the point of view of the load, the
piezoelectric beam acts as a capacitor. Hence, the voltage $w$ follows
the discharge law of a capacitor:
\begin{equation*}
\dot{w}=-\lambda w-\kappa \dot{x},
\end{equation*}
where $\lambda$ is a time constant associated with the capacitance of
the piezoelectric beam and the resistance of the load, and $\kappa>0$
is the electrical piezoelectric constant.  However, in such a
configuration, a mechanical auto-coupling effect occurs: the beam sees
its own generated voltage $w$ and the piezoelectric properties of the
beam generates a strain opposite to the currently applied one.  This
not only has a dissipative effect, as it slows down the beam, but also
increases the dimension of the system by one (see~\cite{ErtHofInm09}).
The system then
becomes:
\begin{align*}
&\ddot{x}+2\zeta
\dot{x}-\frac{1}{2}x(1-x^2)-\chi w=\varepsilon G(t)\\
&\dot{w}=-\lambda w -\kappa \dot{x},
\end{align*}
where $\chi>0$ is the mechanical piezoelectric constant.

The length of a piezoelectric beams or cantilevers plays a crucial
role in the efficiency of the energy harvesting system, as it
determines the frequency of the external forcing, $\varepsilon G(t)$,
for which the device is ``optimal''.  Therefore, such devices need to
be designed to resonate at a particular frequency. A big effort has
been done from the design point of view to broaden this bandwidth. A
common approach, introduced in~\cite{KimJunLeeJan11}, is to consider
coupled oscillators of different lengths such that the device exhibits
different voltage peaks at different frequencies. Other approaches
consider different structural configurations (\cite{ErtRenInm09}) to
achieve a similar improvement, or study the number of piezoelectric
layers connected in different series-parallel configurations
(\cite{ErtInm09}). However, mathematical studies of those models seem
relegated to numerical simulations and bifurcation
analysis~\cite{Feretal09,StaOweMan12,VocNerTraGam12}. As it was
unveiled in~\cite{MooHol79}, there exist very interesting dynamical
phenomena already in the most simple case of a single beam under a
periodic forcing (as in Figure~\ref{fig:damped_beam}), such as
homoclinic tangles, horseshoes and a Duffing chaotic attractor; also,
when neglecting the damping, KAM theorem holds providing the existence
of invariant curves. These, in the absence of dissipation, are
boundaries in the state space and hence act as energy bounds.
Therefore, assuming an external forcing of $O(\varepsilon)$, the
amplitude of the oscillations of the beam cannot grow beyond this
order hence restricting the amount of energy that can be absorbed from
the source.\\
\begin{figure}
\begin{center}
\begin{picture}(1,0.4)
\put(0,0){
\subfigure[\label{fig:coupled_piezo}]
{\includegraphics[width=0.4\textwidth]{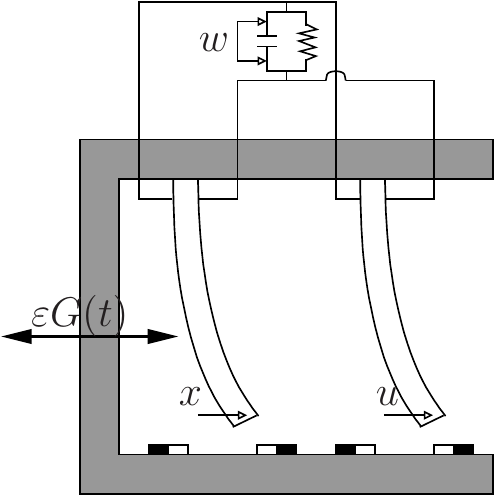}}
}
\put(0.51,0){
\subfigure[\label{fig:coupled_piezo_spring}]
{\includegraphics[width=0.4\textwidth]{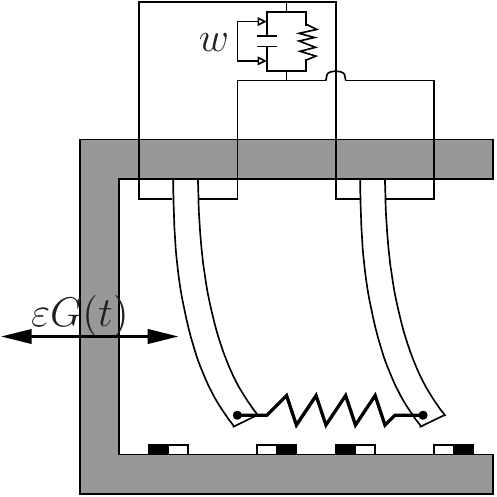}}
}
\end{picture}
\end{center}
\caption{(a) Two piezoelectric oscillators coupled through the
electrci circuit. (b) Two piezoelectric oscillators with additional
conservative coupling (a spring).}
%\label{fig:coupled_piezo}
\end{figure}%
Even more interesting from the dynamical point of view is the higher
dimensional case when considering two or more coupled damped
oscillators. A common approach is to couple them in parallel as in
Figure~\ref{fig:coupled_piezo}, although series connection is also
used (see~\cite{ErtInm09} for a comparison between parallel and series
connection of piezoelectric layers). When connected as in
Figure~\ref{fig:coupled_piezo}, the piezoelectric beams become
mechanically coupled
through the piezoelectric coupling effect: the voltage generated by
one beam accelerates or slows down the other one through the electric
circuit. This can be modeled with the following equations
(\cite{Litetal11})
\begin{equation}
\begin{aligned}
&\ddot{x}+2\zeta
\dot{x}-\frac{1}{2}x(1-x^2)-\chi w=\varepsilon G(t)\\
&\ddot{u}+2\zeta
\dot{u}-\frac{1}{2}u(1-u^2)-\chi w=\varepsilon G(t)\\
&\dot{w}=-\lambda w-\kappa \left( \dot{x}+\dot{u} \right).
\end{aligned}
\label{eq:model_nospring}
\end{equation}
Note that the only coupling term, $\chi w$, is dissipative.\\
When neglecting dissipative terms ($\zeta=\chi=0$), the dimension of
KAM tori is not large enough to act as energy bounds and one may
observe {\em Arnold diffusion} (\cite{Arn64}): existence of
trajectories exhibiting $O(1)$ growth in their ``energy'' when the
device is driven by an arbitrarily small periodic forcing
($0<\varepsilon\ll 1$).  Therefore, if oscillators are conservatively
coupled, the phenomenon of Arnold diffusion could help such devices to
exhibit robustness to the frequency of the periodic source and higher
efficiency than acting separately. Hence, in order to increase the
chances of taking advantage of this phenomenon, we propose to
introduce a conservative coupling between the oscillators.
Physically, such coupling can be achieved by introducing a spring
linking the two beams, as in Figure~\ref{fig:coupled_piezo_spring}.
Assuming that beams are equal and that the displacement of their lower
ends is only horizontal, the spring is kept horizontal. In this case,
the model becomes
\begin{equation}
\begin{aligned}
&\ddot{x}+2\zeta
\dot{x}-\frac{1}{2}x(1-x^2)-\chi w-k(u-x)=\varepsilon G(t)\\
&\ddot{u}+2\zeta
\dot{u}-\frac{1}{2}u(1-u^2)-\chi w-k(x-u)=\varepsilon G(t)\\
&\dot{w}=-\lambda w-\kappa \left( \dot{x}+\dot{u} \right),
\end{aligned}
\label{eq:model_spring}
\end{equation}
where $k$ is the elastic constant of the spring. These equations are
also obtained when linearising around the horizontal position of the
beam.

Arnold diffusion was introduced in the celebrated paper of
Arnold~\cite{Arn64}. Recently, researchers have achieved impressive
advances providing rigorous results to prove the existence of such
trajectories in general Hamiltonian
systems~\cite{BerKalZha16,GidLlaSea14,Mar16c}. The most paradigmatic
applications of Arnold diffusion are associated with classical
problems in celestial mechanics such as instabilities in the {\em restricted three-body
problem} or the Kirkwood gaps in the asteroids belt of the solar
system~\cite{FejozGKR16}, although it has also been proven in physical
examples such as ABC magnetic fields~\cite{LuqPer15}. Partial results
have also been proven in mechanical systems with
impacts~\cite{GraHogSea14}.\\
The main mechanism for diffusion is based on the existence of normally
hyperbolic invariant manifolds (NHIMs) containing the mentioned KAM
tori. By combining {\em inner} dynamics in these manifolds and
{\em outer} through homoclinic/heteroclinic excursions, such tori can
be skipped allowing the trajectories to further accumulate energy from
the source. The study of these outer excursions was greatly
facilitated by the introduction of the {\em Scattering}
map~\cite{DelLlaSea06,DelLlaSea08}.

Unfortunately, theory for Arnold diffusion is still too restrictive to
be applied in systems of the types~\eqref{eq:model_nospring}
and~\eqref{eq:model_spring}, mainly due to the presence of
dissipation, as it provides an extra obstacle to the growth of
energy.\\
In this work we present a first step on the study of Arnold diffusion
in energy harvesting systems based on damped oscillators. In
particular, we focus on a system based on the coupling of two
piezoelectric beams as in Figure~\ref{fig:coupled_piezo_spring} and we
perform a theoretical and numerical study of the invariant objects,
their dynamics and their connections by means of the parameterization
method. These objects play a crucial role in the known mechanisms for Arnold
diffusion, given by combination of dynamics close to Normally
Hyperbolic Invariant Manifolds (NHIM's) ({\em inner dynamics}) and
{\em homoclinic excursions} along the intersection of their stable and
unstable manifolds. In this article we have a less
ambitions goal and we perform a first step in this direction: we
perform a theoretical study of the existence and persistence of a
NHIM, its numerical computation as well as its inner dynamics and its
stable and unstable manifolds by means of the so-called {\em
parameterization method}~\cite{CabFondlL05,HarCanFigLuqMon16}.  A
detailed study of homoclinic intersections and the Scattering map is
left for a future work. The main difficulty relies on the dimension of
the system, which is $6$-dimensional and the presence of dissipation
in both the oscillators (through damping) and the coupling (inverse
piezoelectric effect).\\

This work is organized as follows. In
Section~\ref{sec:invariant_objects} we introduce a generalized version
of the system in a perturbation setting: forcing, dissipation and
coupling are included only in $O(\varepsilon)$ terms. We analyze its
invariant objects for the unperturbed case and their persistence after
introducing the perturbation. In
Section~\ref{sec:parametrization_method} we present the theoretical
setting necessary to apply the Newton-like method introduced
in~\cite{HarCanFigLuqMon16} based on the parameterization method. In
Section~\ref{sec:numerical_results} we present the obtained numerical
results, studying the inner dynamics for different configurations
regarding the two types of dissipations (damping and piezoelectric
coupling).  Finally, we present our conclusions in
Section~\ref{sec:conclusions}.

\section{Invariant objects}\label{sec:invariant_objects}
%\subsection{Invariant objects of the unperturbed system}\label{sec:unper_syst}
\subsection{Generalization of the model}
As mentioned in the introduction, this paper is concerned with the
study of invariant manifolds of a particular energy harvesting system
consisting of two coupled piezoelectric beams. However, many of the
results and techniques that we show are valid for a larger class of
systems. Hence, in this section we introduce a general class of
systems that for which our results hold. We first consider a
Hamiltonian system of the form
\begin{equation}
H_\varepsilon(x,y,u,v,s)=\X(x,y)+\U(u,v)+\varepsilon h(x,y,u,v,t),
\label{eq:hamiltonian_general}
\end{equation}
with $h(x,y,u,v,t+T)=h(x,y,u,v,t)$ and $\varepsilon\ge0$ a small
parameter. Assume that
\begin{enumerate}[{\em h1}]
\item the system associated with the Hamiltonian $\X$ possesses a
saddle point, $Q_0$, with a homoclinic loop, $\gamma$, parameterized by
a function $\sigma(t):\RR\to\RR^2$:
\begin{equation*}
\gamma=\W^s(Q_0)=\W^u(Q_0)=\left\{ \sigma(t)\in\RR^2,\,t\in\RR \right\}
\end{equation*}
satisfying
\begin{equation}
\sigma'(t)=
\left( 
\begin{array}[]{cc}
0&1\\-1&0
\end{array}
 \right)
\nabla \X(\sigma(t)),
\label{eq:sigma_is_flow}
\end{equation}
($\sigma(t)$ is a solution of the Hamiltonian $\X$) and
\begin{equation*}
\lim_{t\to\pm\infty}\left| \sigma(t)-Q_0
\right|<\lim_{t\to\pm\infty}Me^{-\tilde{\lambda}\left| t \right|}=0,
\end{equation*}
for some $M,\tilde{\lambda}>0$,
\item the system associated with the Hamiltonian $\U$ is integrable in
some open set in the Liouville sense (can be written in action-angle
variables). Moreover, it satisfies the twist condition (associated
frequencies of its invariant sets are monotone).
\end{enumerate}
\begin{remark}
Alternatively, condition {h2} can be stated as follows: ``The
system associated with Hamiltonian $\U$ possesses a continuum of
periodic orbits, $\U(u,v)=c$, whose periods are monotone in $c$''.
\end{remark}
\begin{remark}\label{rem:action-angle}
One could assume that $\U(u,v)$ is given in action-angle variables:
$U(I,\phi)=\Omega(I)$ with $\Omega'(I)\neq 0$. These canonical
variables would of course simplify the notation in the theoretical
statements. However, in applications, one frequently finds systems
that are integrable but are not given in such variables (as it is our
case). Provided that this change of variables becomes difficult to
explicitly apply, we prefer to keep a general Hamiltonian $U(u,v)$ in
order to allow applications to deal with original variables as much as
possible.\\
However, in Section~\ref{sec:unperturbed_objects}, it will be useful
to introduce a parameterization introducing action-angle-like
variables, which can be easily numerically compute.
\end{remark}
\begin{remark}
Similarly as in Remark~\ref{rem:action-angle}, the Hamiltonian
$X(x,y)$ could be assumed to be a pendulum:
$X(x,y)=\frac{y^2}{2}+\cos(x)-1$. For the same reason we keep here a
general Hamiltonian $X(x,y)$.
\end{remark}
To System~\eqref{eq:hamiltonian_general} we add a small dissipative
coupling given by an extra equation leading to the $5$-dimensional
non-autonomous system
\begin{equation}
\begin{aligned}
\dot{z}&=J_1 \nabla\X(u,v)+J_2 \nabla\U(x,y)\\
&\quad+\varepsilon \Big(J_3 \nabla h(z,t)+\nu g(z,w)\Big)\\
\dot{w}&=-\lambda w +b(z),
\end{aligned}
\label{eq:hamiltonian_with_dissipation}
\end{equation}
where $z=(x,y,u,v)$,
\begin{equation*}
J_1=\left(
\begin{array}[]{cc}
0&1\\-1&0\\0&0\\0&0
%\\0&0
\end{array}
 \right),
\quad
J_2=\left(
\begin{array}[]{cc}
0&0\\0&0\\0&1\\-1&0
%\\0&0
\end{array}
 \right)
\end{equation*}
and
\begin{equation*}
J_3=\left(
\begin{array}[]{cccc}
0&1&0&0\\
-1&0&0&0\\
0&0&0&1\\
0&0&-1&0
%\\0&0&0&0
\end{array}
 \right).
\end{equation*}
While $\varepsilon$ is a perturbative parameter
($0\le\varepsilon\ll1$), $\nu$ is not necessary small. The latter
allows to couple and uncouple system $z$ with $w$, distinguishing
between a conservative and dissipative behaviour regarding coordinates
$z$.

By adding time as a variable, $t=s\in\Tt=\RR/T\mathbb{Z}$, and calling
$\tz=(z,w,s)\in\RR^5\times\Tt$, we write
system~\eqref{eq:hamiltonian_with_dissipation} in an autonomous and
more compact form as
\begin{equation}
\dot{\tz}=g_0(\tz)+\varepsilon\left(g_1(\tz)+\nu g_2(\tz)\right)
\label{eq:pert_syst}
\end{equation}
with
\begin{equation*}
g_0(\tz)=\left( 
\begin{array}[]{c}
J_1\nabla\X(x,y)+J_2\nabla\U(u,v)\\-\lambda w+b(z)\\1
\end{array}
 \right),
\end{equation*}
\begin{equation*}
g_1(\tz)=\left( 
\begin{array}[]{c}
g(z,w)\\0\\ 0
\end{array}
 \right)
\end{equation*}
and
\begin{equation*}
g_2(\tz)=\left( 
\begin{array}[]{c}
J_3\nabla h(x,y,v,s)\\0\\0
\end{array}
 \right).
\end{equation*}
Note that $g_1$ contains the dissipative terms and coupling while
$g_2$ the conservative coupling and forcing.

\subsection{Invariant objects of the unpertubed system}\label{sec:unperturbed_objects}
For $\varepsilon=0$, the unperturbed system~\eqref{eq:pert_syst} becomes
\begin{equation}
\dot{\tz}=g_0(\tz),
\label{eq:unpert_syst}
\end{equation}
which consists of the crossed product of systems $(x,y)$, $(u,v)$,
$w$ and $s$.\\
As system $\U$ is integrable, it possesses periodic orbits given by
\begin{equation}
\Po^c=\left\{ (u,v)\in\RR^2,\,|\,\U(u,v)=c \right\},
\label{eq:periodic_orbits}
\end{equation}
whose period, $T_c$, is monotone with $c$ due to the twist condition.
Assume $T'_c>0$.\\
In order to construct a Normally Hyperbolic Invariant Manifold for
system~\eqref{eq:unpert_syst}, we focus on these periodic orbits for
system $\U$ while system $\X$ remains at the hyperbolic point
$Q_0=(0,0)$.  Provided that each of these periodic orbits is contained
in the energy level given by $\U(u,v)=c$, as mentioned in
Remark~\ref{rem:action-angle}, it will be useful to parametrize them
by $c$ and an angle, $\theta$. Hence we will consider ``action
angle''-like variables $(\theta,c)\in \T\times \RR$ to parametrize
periodic orbits $\Po^c$ as follows.  Let $\varphi_\U(t;u_0,v_0)$ be
the flow associated with system $\U$, such that
$\varphi_\U(0;u_0,v_0)=(u_0,v_0)$. Choose a section in $\RR^2$
transversal to all periodic orbits $\U(u,v)=c$, and let $(u_c,v_c)$
the point in that section at level of energy $c$. Then we consider the
parameterization
\begin{equation}
\begin{array}{cccc}
\p:&\T\times \RR&\longrightarrow&\RR^2\\
&(\theta,c)&\longmapsto&\varphi_\U(\theta T_c;u_c,v_c).
\end{array}
\label{eq:change_variables}
\end{equation}
\iffalse
Note that this induces the dynamics
\begin{align*}
\dot{\theta}&=\frac{1}{T_c}\\
\dot{c}&=0.
\end{align*}
\fi
As for action-angle variables, this change of variables can be
difficult (or impossible) to apply explicitly. However, as we will
show in Sections~\ref{sec:unperturbed_bundles}
and~\ref{sec:numerical_results}, it can be easily numerically
implemented.

The following lemma gives as the existence of a Normally Hyperbolic
Invariant Manifold for System~\eqref{eq:pert_syst} when
$\varepsilon=0$.
\begin{lemma}
\begin{enumerate}[i)]
\item System~\eqref{eq:unpert_syst} possesses a foliated $3$-dimensional
invariant \linebreak ma\-nifold
\begin{equation*}
\tL=Q_0\times \bigcup_{c_1\le c\le c_2}\tL^c
\end{equation*}
with
\begin{align}
\tL^c=&\Big\{
(u,v,w,s)\in\Po^c\times\RR\times\Tt,\,|\,w=w^p_0(u,v)
\Big\}\nonumber\\
\Po^c&=\Big\{(u,v)\in\RR^2,\,|\,\U(u,v)=c \Big\}\nonumber\\
w^p_0(u,v)&=\frac{e^{-\lambda
T_c}}{1-e^{-\lambda T_c}}\int_0^{T_c}
b(Q_0,\varphi_\U(s;u,v))ds\label{eq:w0p}
\end{align}
These objects can be written by means of the parameterizations
\begin{equation*}
\begin{array}[]{cccc}
\p:&\T\times(0,\infty)&\longrightarrow&\RR^2\\
&(\theta,c)&\longmapsto&\varphi_\U(\theta T_c;u_c,v_c)
\end{array}
\end{equation*}
\begin{equation*}
\begin{array}{cccc}
\tK^c:&\T\times\Tt&\longrightarrow&\RR^3\times\Tt\\
&(\theta,s)&\longmapsto&\left( p(\theta,c),w^p_0(p(\theta,c)),s \right)
\end{array}
\end{equation*}
\begin{equation*}
\begin{array}{cccc}
\tK:&\T\times \RR\times\Tt&\longrightarrow& \RR^5\times
\Tt\\
&(\theta,c,s)&\longmapsto&\left( Q_0,\tK^c(\theta,s) \right)
\end{array}
%\label{eq:parameterization_tLambda}
\end{equation*}
such that
\begin{align*}
\Po^c&=\p(\T,c)\\
\tL^c&=\tK^c(\T,\Tt)\\
\tL&=\tK(\T,[c_1,c_2],\Tt).
\end{align*}
\item The manifold $\tL$ is normally hyperbolic, has a
$5$-dimensional stable manifold
\begin{equation*}
\W^s(\tL)=\W^s( Q_0)\times \bigcup_{c\in[c_1,c_2]}\Po^c\times
\RR\times\Tt
\end{equation*}
and a $4$-dimensional unstable manifold forming a
homoclinic manifold,
\begin{equation*}
\tG=\W^u(\tL)\subset \W^s(\tL),
\end{equation*}
where
\begin{align*}
\tG&=\bigcup_{c_1\le c\le c_2}\Big\{\left(\sigma(\tau),u,v,w,s
\right)\in \Po^c\times\gamma\times\RR\times\Tt,\\
&\qquad \tau\in\RR,\,w=w^u_0(\tau,u,v)\Big\}\\
\end{align*}
and
\begin{align*}
w^u_0(\tau,u,v)&=\int_{-\infty}^0
\Big(b\big( \sigma(\tau+s),\varphi_\U(s;u,v) \big)-
b\big(Q_0, \varphi_\U(s;u,v)\big)\Big)e^{\lambda s}ds\nonumber\\
&+w^p_0(u,v).
%\label{eq:condini_wu}\\
\end{align*}
The unstable and stable manifolds can be parameterized by
\begin{align*}
\W^s(\tL)&=\tW^{s}(\T,[c_1,c_2],\Tt,\RR^2)\\
\W^u(\tL)&=\tW^{u}(\T,[c_1,c_2],\Tt,\RR),
\end{align*}
\iffalse
\begin{align*}
\W^s(\tL)&=\tW^{s,+}(\T,[c_1,c_2],\Tt,\RR^2)\\
&\qquad\cup \tW^{s,-}(\T,[c_1,c_2],\Tt,\RR^2)\\
\W^u(\tL)&=\tW^{u,+}(\T,[c_1,c_2],\Tt,\RR)\\
&\qquad\cup \tW^{u,-}(\T,[c_1,c_2],\Tt,\RR)
\end{align*}
\fi
where
\begin{equation*}
\begin{array}{cccc}
\tW^{s}:&\T\times[c_1,c_2]\times\Tt\times\RR^2&\longrightarrow&\RR^5\times\Tt\\
&(\theta,c,s,\tau,r)&\longmapsto&\left(
\sigma(\tau),\p(\theta,c),r,s \right)
\end{array}
\end{equation*}
\begin{equation*}
\begin{array}{cccc}
\tW^{u}:&\T\times[c_1,c_2]\times\Tt\times\RR&\longrightarrow&\RR^5\times\Tt\\
&(\theta,c,s,\tau)&\longmapsto&\left(
\sigma(\tau),\p(\theta,c),w^u_0(\tau,\p(\theta,c)),s \right)
\end{array}
\end{equation*}

\end{enumerate}
\label{lem:NHIM_general}
\end{lemma}
\begin{proof}
When $(x,y)=Q_0$, provided that $\lambda>0$, the variable $w$ is
attracted to a certain object given by the dynamics of $v$. Given
$(u_0,v_0)$ we define 
\begin{equation}
b^p(t;u_0,v_0)=b\left(Q_0,\varphi_\U(t;u_0,v_0)\right).
\label{eq:bp}
\end{equation}
Then the dynamics of $w$ restricted to $(x,y)=Q_0$ is given by
\begin{equation}
%\dot{w}=-\lambda w-kv^p(t).
\dot{w}=-\lambda
w+b^p(t;u_0,v_0).
\label{eq:w-dynamics}
\end{equation}
Provided that Equation~\eqref{eq:w-dynamics} is linear in $w$ and
$b^p(t;u_0,v_0)$ is $T_c$-periodic, System~\eqref{eq:w-dynamics}
possesses an $T_c$-periodic orbit:
\begin{equation*}
w^p(t+T_c)=w^p(t),
\end{equation*}
which is attracting (because $\lambda>0$).\\
\iffalse
This periodic orbit can be
parametrized as a function of $u$ as follows. Performing the change of
variable given by
\begin{equation*}
\frac{dw}{dt}=\frac{dw}{du}\dot{u}(t)=\frac{dw}{du}v=\frac{dw}{du}\sqrt{2c+\frac{1}{2}u^2(1-u^2)}
\end{equation*}
we rewrite Equation~\eqref{eq:w-dynamics} as
\begin{equation}
\frac{dw}{du}+P(u)w=-k,
%\label{eq:w2}
\end{equation}
with
\begin{equation*}
P(u)=\frac{\lambda}{\sqrt{2c+\frac{1}{2}u^2(1-u^2)}}.
\end{equation*}
The solution of Equation~\eqref{eq:w2} becomes
\begin{equation*}
\varphi_w(u;u_0)=e^{-\int_{u_0}^uP(s)ds}\left( -k \int_{u_0}^u
e^{\int_{u_0}^sP(r)dr}ds+w(u_0;u_0) \right).
\end{equation*}
We now find the initial condition for a periodic orbit:
\begin{equation*}
\varphi(u_{1};u_0)=\varphi(u_0;u_0)
\end{equation*}
\COMMENT{I ara qu\`e?}

When combining  the periodic orbits $\Lambda_c$ for system $\U$ and
the hyperbolic point $Q_0$ for system $\X$ we obtain the $3$-dimensional
Normally Hyperbolic Invariant Manifold (NHIM)
\begin{equation*}
\tL=Q_0\times \left(\bigcup_{c_1\le c\le c_2}
\Big\{ \left(u,v,w,s\right)\,|\,\U(u,v)=c,\,s\in\Tt,\,w=\sigma(u;s)
\Big\}
\right)\subset
\RR^5\times \Tt.
\end{equation*}
\COMMENT{Other option:}\\
{\color{red}
\fi
We compute the initial condition for such periodic orbit.  Note that,
although system~\eqref{eq:w-dynamics} is not autonomous, we can assume
that the intial conditions are given for $t=0$, since
Equation~\eqref{eq:w-dynamics} has to be integrated together with the
equations for $\dot{u}$ and $\dot{v}$, which provides a
$3$-dimensional autonomous system.  Therefore, the general solution
of~\eqref{eq:w-dynamics} becomes 
\iffalse
\begin{equation*}
w^p(t;t_0,w^p_0)=e^{-\lambda(t-t_0)}\left(\int_{0}^t-k v^p(t;t_0,v_0)e^{\lambda
(s-t_0)}ds+w_p(0)\right),
\end{equation*}
\fi
\begin{equation*}
w(t;w_0,u_0,v_0)=e^{-\lambda t}\left(\int_{0}^tb^p(s;u_0,v_0)e^{\lambda
s}ds+w_0\right),
\end{equation*}
from where, imposing $w(T_c;w_0)=w_0$, we get that the initial
condition for a periodic orbit is
\begin{equation*}
w^p_0(u_0,v_0)=\frac{e^{-\lambda
 T_c}}{1-e^{-\lambda T_c}}\int_0^{T_c}b^p(s;u_0,v_0)e^{\lambda s}ds.
\end{equation*}
Hence, given $(u_0,v_0)$, the attracting periodic orbit of $w$ becomes
\begin{equation}
\begin{aligned}
w^p(t)&=e^{-\lambda t}\Bigg( \int_0^tb^p(s;u_0,v_0)e^{\lambda
s}ds\\
&\quad+\frac{e^{-\lambda
T_c}}{1-e^{-\lambda T_c}}\int_0^{ T_c}b^p(s;u_0,v_0)e^{\lambda
s}ds \Bigg).
\end{aligned}
\label{eq:w_po}
\end{equation}
Note that $w^p_0$ depends on $u_0$ and $v_0$ through the periodic
orbit~\eqref{eq:bp}.  Moreover, $w_0^p(u_0,v_0)$ is indeed a
parametrization of the whole periodic orbit $w^p(t)$: just by keeping
$t=0$ and varying $u_0,v_0$ along the periodic orbit $\Po^c$
$w_0^p(u_0,v_0)$ evolves along the periodic orbit $w^p(t)$.
Hence we have obtained the $3$-dimensional invariant manifold
\iffalse
\begin{equation*}
\tL=Q_0\times \left(\bigcup_{c_1\le c\le c_2}
\Big\{ \left(\theta,c,w\right)\,|\,\theta\in \T,\,w=w^p_0(\theta,c)
\Big\}
\right)\times\Tt\subset
\RR^2\times \T\times[c_1,c_2]\times\RR\times \Tt.
\end{equation*}
which, in the original variables becomes
\fi
\begin{equation*}
\tL=Q_0\times \left(\bigcup_{c_1\le c\le c_2}
\Big\{ \left(u,v,w\right)\,|\,\,\U(u,v)=c,\,w=w^p_0(u,v)
\Big\}
\right)\times\Tt\subset
\RR^5\times \Tt.
\end{equation*}
\iffalse
\begin{equation*}
\tL=Q_0\times \bigcup_{c_1\le c\le c_2}\left(
\Big\{ \left(u,v\right)\,|\,\U(u,v)=c\Big\}
\times\Big\{w_p(t;u,v),\,t\in[0,T_c]\Big\}
\right)
\times \Tt
\subset
\RR^5\times \Tt.
\end{equation*}
\fi
%}
Recalling that $(u_0,v_0)$ can be parametrized by $\p(\theta,c)$ as in
Equation~\eqref{eq:change_variables},
$w_0^p$ can be as well parametrized by $\theta$ and $c$:
$w_0^p=w_0^p(\p(\theta,c))$. This induces a parametrization for $\tL$
\begin{equation*}
\tK:\T\times \RR\times\Tt\longrightarrow \RR^5\times
\Tt,
%\label{eq:parameterization_tLambda}
\end{equation*}
given by
\begin{equation*}
\tK(\theta,c,s)=
\left( 
\begin{array}[]{c}
%0\\0\\\theta\\c\\w^p_0(\theta,c)\\s
0\\0\\
\p(\theta,c)\\
w^p_0(\p(\theta,c))\\s
\end{array}
 \right)
%\label{eq:parameterization_explicit}
\end{equation*}
and hence
\begin{equation*}
\tL=\tK(\T,[c_1,c_2],\Tt).
\end{equation*}
The invariant manifold $\tL$ is foliated by $2$-dimensional invariant
tori contained at the energy level $c$:
\begin{equation*}
%\tL=Q_0\times\bigcup_{c_1\le c\le c_2}\tTau_c.
%\tL=Q_0\times\bigcup_{c_1\le c\le c_2}\torus_c.
\tL=Q_0\times\bigcup_{c_1\le c\le c_2}\tL^c.
\end{equation*}
Each of these tori is homeomorphic to $\T\times\Tt$,
\begin{equation*}
%\tTau_c\simeq\T\times\Tt,
%\torus_c\simeq\T\times\Tt,
\tL^c\simeq\T\times\Tt,
\end{equation*}
as it can be parametrized by $(\theta,s)$:
\begin{equation*}
%\torus_c=\tilde{L}^c \left(\T,\Tt\right),
%\tTau_c=\tilde{T}^c \left(\T,\Tt\right),
%\torus_c=\tK\left(\T,c,\Tt\right),
%\tL^c=\tK\left(\T,c,\Tt\right),
\tL^c=\tK^c\left(\T,c,\Tt\right),
\end{equation*}
where
\begin{equation*}
\tK^c:\T\times\Tt\longrightarrow\RR^3\times\Tt
%\label{eq:parameterization_tori}
\end{equation*}
and
\begin{align*}
\tK^c(\theta,s)&=\Pi_{u,v,w,s}\left(\tK(\theta,c,s)\right)\\
&=(p(\theta,c),w_0^p(p(\theta,c)),s).
\end{align*}

We now show that the invariant manifold $\tL$ is normally hyperbolic
by showing that it has stable/unstable normal bundles with exponential
convergence/divergence.\\
By fixing coordinates $(\theta,c,s)$, we focus on a point at
the manifold $\tL$,
\begin{equation}
K(\theta,c,s)=\tz^b=(Q_0,u,v,w_0^p(u,v)),s)\in \tL\label{eq:restriction3},
\end{equation}
and we study its stable and unstable fibers.

Clearly, hyperbolicity is inherited from the hyperbolic point $Q_0$.
Hence, coordinates $(x,y)$ of points of the invariant fibers of
$\tz^b$ are given by the homoclinic loop of $Q_0$, parametrized by
$\tau$:
\begin{equation}
\sigma(\tau)\in \gamma.
\label{eq:restriction1}
\end{equation}
Coordinates $(u,v,s)$ need to be equal those of $\tz^b$ due to their
lack of hyperbolicity. So it only remains to find proper values of $w$
to define the stable and unstable fibers of $\tz^b$.

Letting $\varphi_\X(t;x,y)$ be the flow associated with the
Hamiltonian $\X$ we define
\begin{align*}
b^h(t;\tau,u,v)&=b\Big(\varphi_\X\big(t;\sigma(\tau)\big),\varphi_\U\big(t;u,v\big)\Big)\\
&=b(\sigma(\tau+t),\varphi(t;u,v)).
\end{align*}
The last equality holds due to condition given in
Equation~\eqref{eq:sigma_is_flow}.\\
For $(x,y)\in \gamma$, the variable $w$ evolves following the equation
\begin{equation*}
\dot{w}=-\lambda w+b^h(t;\tau,u,v),
%\label{eq:w_unpertubed_full}
\end{equation*}
which has the general solution
\begin{equation}
w(t;w_0)=e^{-\lambda t}\left(w_0+\int_{0}^tb^h\left(s;\tau,u,v)\right)e^{\lambda
s}ds\right).
\label{eq:sol_w_homo}
\end{equation}
Then, the values $w_0^s$ and $w_0^u$ that we are looking need to satisfy
\begin{equation*}
\lim_{t\to \infty}\left|w(t;w^s_0)-w(t;w_0^p)  \right|\longrightarrow 0
\end{equation*}
and 
\begin{equation*}
\lim_{t\to-\infty}\left|w(t;w^u_0)-w(t;w_0^p)  \right|\longrightarrow
0.
\end{equation*}
We define
\begin{equation*}
z(t;z_0)=w(t;w_0)-w(t;w_0^p),
\end{equation*}
with
\begin{equation*}
z_0=w_0-w^p_0.
\end{equation*}
Defining
\begin{equation*}
b^z(t;\tau,u,v)=b^h(t;\tau,u,v)-b^p(t;u,v),
\end{equation*}
$z(t)$ becomes
\begin{equation*}
%z(t)=w(t)-w^p(t)=e^{-\lambda t}\left(z_0-\int_{0}^tky^h(s)e^{\lambda
z(t)=e^{-\lambda
t}\left(z_0+\int_{0}^tb^z(s;\tau,u,v)e^{\lambda
s}ds\right).
\end{equation*}
Due to the hyperbolicity of $Q_0$ and the fact that $b(x,y,u,v)$ is
continuous, we know that there exist positive constants $\tl$, $M$ and
$\delta$ such that
\begin{equation*}
\left| b^z(t;\tau,u,v)\right|<M e^{-\tl \left| t \right|}
\end{equation*}
if $\left| t \right|>\delta$.\\
On one hand, we get that
\begin{align*}
\lim _{t\to+\infty} \left|z(t)  \right|&=
\lim_{t\to+\infty}\left| e^{-\lambda
t}z_0+\int_0^tb^z(s;\tau,u,v)e^{\lambda
(s-t)}ds\right|\\
&<\lim _{t\to+\infty}M\int_0^t e^{-\tl s+\lambda (s- t)}ds=0,
\end{align*}
for any $z_0\in\RR$. As a consequence, all initial conditions
$w_{0}$ are attracted to the periodic orbit $w^p(t)$. Hence, the
stable fiber leaves $w$ free.\\
On the other hand, the limit for $t\to-\infty$ diverges unless we choose
\begin{equation*}
%z_0=-\int_{-\infty}^0ke^{\lambda s}y^h(s)ds.
z_0(\tau,u,v)=\int_{-\infty}^0b^z(s;\tau,u,v)e^{\lambda s}ds.
\end{equation*}
In this case, we get
\begin{align*}
\lim_{t\to-\infty}\left| z(t) \right|&=\lim_{t\to-\infty}\left|
e^{-\lambda t}\Bigg(\int_{-\infty}^0e^{\lambda s}b^z(s)ds+\int_0^te^{\lambda
s}b^z(s)ds\Bigg)
\right|\\
&=\lim_{t\to-\infty}\left| e^{-\lambda t}\int_{-\infty}^te^{\lambda
s}b^z(s)ds\right|\\
&<\lim_{t\to-\infty}\left|e^{-\lambda t} \int_{-\infty}^t
M e^{(\lambda+\tl)s}ds \right|\\
&=\lim_{t\to-\infty}\frac{M e^{\tl t}}{\lambda+\tl}=0.
\end{align*}
Therefore, the unstable fiber of $\tz^b$ is given by points
$(\sigma(\tau),u,v,w_0^u,s)$ with
\begin{equation*}
\begin{aligned}
w_0^u(\tau,u,v)&=\int_{-\infty}^0b^h(s;\tau,u,s)e^{\lambda s}ds+w_0^p(u,v)\\
&=\int_{-\infty}^0\Big(b\big(\sigma(\tau+s),\varphi_\U(s;u,v)\big)\\
&\qquad-b(Q_0,\varphi_\U(s;u,v))\Big)e^{\lambda
s}ds+w_0^p(u,v).
\end{aligned}
\end{equation*}
\end{proof}
\begin{remark}
When $c$ is such that $T$ and $T_c$ are congruent, then $\tL^c$
is filled by periodic orbits: each point is a periodic point of the
$T$-time return map.  However, when $T$ and $T_c$ are
inconmensurable, $\tL^c$ is  densily filled by the trajecteory of any 
point at $\tL^c$.  Note that this implies that, for any
initial condition at the invariant manifold $\tL$, one obtains bounded
dynamics both for $t\to\infty$ and $t\to-\infty$.
\end{remark}
\begin{remark}
The manifolds $\W^s(\tL)$ and $\W^u(\tL)$ generate the normal bundle
to $\tL$, as they generate the $x-y$ plane and the stable manifold
contains the hyperplane $w$.
\end{remark}

\subsection{Persistence of manifolds}\label{sec:persistence}
In order to study the persistence of the manifold $\tL$ for
$\varepsilon>0$, we use theory for normally hyperbolic invariant
manifolds (\cite{Fen72}). However, due to the presence of dissipation
for $\nu>0$, the resulting manifold may lose some properties.
This is summarized in the following
\begin{proposition}
For $\varepsilon>0$ and some $c_0>0$, there exists a unique parameterization
\begin{equation*}
\tK_\varepsilon:\T\times[c_1,c_2]\times\Tt\longrightarrow\RR^5\times\Tt\label{eq:param_Keps}
\end{equation*}
with $0<c_0<c_1<c_2$, such that, the manifold
$\tL_\varepsilon=\tK_\varepsilon(\T,[c_1,c_2],\Tt)$ is unique,
normally hyperbolic and $\varepsilon$-close to $\tK_0=\tK$. Moreover,
\begin{enumerate}[i)]
\item if $\nu=0$, $\tL_\varepsilon$ is invariant and has
boundaries,
\item if $\nu>0$, the manifold
$\tK_\varepsilon(\T,(c_1,c_2),\Tt)$ is locally invariant.
\end{enumerate}
\label{pro:persistence}
\end{proposition}

\begin{proof}
For $\varepsilon>0$ theory of normally hyperbolic invariant manifolds
(\cite{Fen72}) guarantees that (locally) there exists
$\tL_\varepsilon$ $\varepsilon$-close to $\tL$, with $\tL_0=\tL$.\\
If $\nu=0$, the perturbation in System~\eqref{eq:pert_syst} becomes
Hamiltonian and, hence, the dynamics of the system restricted to
$\tL_\varepsilon$ (inner dynamics) becomes symplectic. In this case,
KAM theory (\cite{DelLla00}) provides the existence of invariant tori
in $\tL_\varepsilon$ bounding the inner dynamics. Assuming $c_1$ and
$c_2$ are chosen such that $T_{c_1}/T$ and $T_{c_2}/T$ are
diophantine, the manifold $\tL_\varepsilon$ has boundaries given by
the invariant tori $\tK_\varepsilon(\T,c_1,\Tt)$ and
$\tK_\varepsilon(\T,c_2,\Tt)$. As a consequence, the manifold is
invariant.

When the perturbation includes dissipative terms ($\nu>0$), the
existence of these boundaries is not guaranteed, as KAM tori are
generically destroyed (\cite{CalCelLla13}). Therefore, in this case,
$\tL_\varepsilon$ does not necessary possess boundaries.  However, we
show that it is unique. We consider $[c_1',c_2']$ such that $c_1'>0$
and $[c_1,c_2]\subset[c_1',c_2']$, and construct a new smooth field
$g'_\varepsilon(\tz)$ coinciding with~\eqref{eq:pert_syst}
for $c\in(c_1',c_2')$ and vanishing otherwise. This guarantees the
existence of a ``larger'' normally hyperbolic invariant manifold,
$\tL_\varepsilon'$, $\varepsilon$-close to
$\tK_0(\T,[c_1',c_2'],\Tt)$ and possessing boundaries. Therefore,
theory for normally hyperbolic invariant manifolds holds and
$\tL_\varepsilon'$ is unique and invariant. As
$\tL_\varepsilon\subset \tL_\varepsilon'$ and $g_\varepsilon'$
coincides with $g_0+\varepsilon (g_1+\nu g_2)$ in $\tL_\varepsilon$, the
manifold $\tL_\varepsilon$ is also unique.\\
Although, due to the dissipation, inner dynamics
contains attractors, trajectories may leave $\tL_\varepsilon$ both when
flown forwards or backwards in time.  This however occurs
slowly and points away from original boundaries $c=c_1$ and $c=c_2$
remain in $\tL_\varepsilon$ for large periods of time. Therefore, the manifold
$\tK_\varepsilon(\T,(c_1,c_2),\Tt)$ becomes only locally invariant.

\end{proof}
\begin{remark}
These parameterizations, together with the dynamics of the system 
restricted to $\tL_\varepsilon$ and linear approximations of the
manifolds $\W_\varepsilon^{s}(\tL_\varepsilon)$ and
$\W_\varepsilon^u(\tL_\varepsilon)$ will be numerically computed in
Section~\ref{sec:parametrization_method} by means of the
parameterizaiton method.
\end{remark}
\begin{remark}
The constant $c_0$ guarantees that $c_1$ is enough isolated from
$c=0$. If this does not occur, then the manifold $\tL_\varepsilon$ may
lose normal hyperbolicity, as the tangent dynamics start competing
with the normal ones when periodic orbits $\Po^c$ are too close to the
homoclinic loop. However, this loss of normal hyperbolicity can be
avoided by considering beams of different lengths leading to different
hyperbolic rates.
\end{remark}
\begin{remark}
From Proposition~\ref{pro:persistence} we also get parameterizations
for the stable and unstable manifolds of $\tL_\varepsilon$:
\begin{align}
\tW^{s,\pm}_\varepsilon&:\T\times[c_1,c_2]\times\Tt\times\RR^2\longrightarrow\RR^5\times\Tt\label{eq:param_Wseps}\\
\tW^{u,\pm}_\varepsilon&:\T\times[c_1,c_2]\times\Tt\times\RR\longrightarrow\RR^5\times\Tt\label{eq:param_Wueps}
\end{align}
\end{remark}
%\section{Numerical computation of $\tL_\varepsilon$ and the inner dynamics}\label{sec:parametrization_method}
\subsection{Two-coupled piezoelectric oscillators}\label{sec:unper_syst}
In this section we apply the previous results to the case of two
coupled piezo-electric oscillators. We first write
System~\eqref{eq:model_spring} as in Equation~\eqref{eq:pert_syst}.\\
Let us assume that the damping ($\zeta$), the piezolectric
coupling ($\chi$) and  the elastic constant of the spring ($k$) are
small. We introduce scalings to write these parameters in terms of the
amplitude of the small forcing as follows
\begin{equation}
\zeta=\varepsilon\tzeta,\quad
\chi=\varepsilon\tc,\quad k=\varepsilon\tk.
\label{eq:scaling}
\end{equation}
The parameter $\nu$ in Equation~\eqref{eq:scaling} is not a real parameter of
the system but it is artificially introduced in order to  allow
distinguishing between a conservative case ($\nu=0$) and the general
dissipative one ($\nu>0$) so. This situation can be distinguished
between $\tzeta=\tc=0$ or $\tzeta>0$ and/or $\tc>0$. Therefore,
when dealing with the real model of coupled piezo-electric oscillators
we will ignore $\nu$.\\
The scalings~\eqref{eq:scaling} allow us to write
System~\eqref{eq:model_spring} in the perturbative form given in
Equations~\eqref{eq:hamiltonian_general}-\eqref{eq:hamiltonian_with_dissipation}
with
\begin{align}
\X(x,y)&=\frac{y^2}{2}-\frac{1}{4}x^2\left( 1-\frac{x^2}{2}
\right)\label{eq:X_piezo}\\
\U(u,v)&=\frac{v^2}{2}-\frac{1}{4}u^2\left( 1-\frac{u^2}{2}\right)\\
h(z,t)&=-\frac{\tk}{2}(u-x)^2-(x+u)G(t)\\
g(z)&=\tc w\\
b(z)&=-\kappa(y+v).\label{eq:b_piezo}
\end{align}
Note that System~\eqref{eq:model_spring} has been reduced to a first
order system be adding the variables $y=\dot{x}$ and $v=\dot{u}$.\\
In the autonomous and more compact form given in
Equation~\eqref{eq:pert_syst}, the functions $g_i$ become
\begin{equation}
g_0(\tz)=
\left( 
\begin{array}{c}
y\\
\frac{1}{2}x(1-x^2)\\ v \\ \frac{1}{2} u(1-u^2)\\-\lambda w-\kappa(y+v)\\1
\end{array}
 \right)
\label{eq:g0_piezo}
\end{equation}
\begin{equation}
g_1(\tz)=\left( 
\begin{array}[]{c}
0\\ -2\tzeta y+\tc w \\ 0\\ -2\tzeta v +\tc w\\0\\0
\end{array}
 \right)
\label{eq:perturbation}
\end{equation}
and
\begin{equation}
g_2(\tz)=\left( 
\begin{array}[]{c}
0\\ k(u-x)+G(s)\\0\\k(x-u)+G(s)\\0\\0
\end{array}
 \right).
\label{eq:g2_piezo}
\end{equation}

For $\varepsilon=0$, the $(x,y,u,v)$ system becomes two uncoupled and
unforced beams modeled by the Hamiltonians $\X$ and $\U$. In this
case, these Hamiltonians are equal, but satisfy conditions~\condsh{}.
The phase portrait for each Hamiltonian is shown in
Figure~\ref{fig:figure_eight} and consists of a figure of eight.
\begin{figure}
\begin{center}
\includegraphics[angle=-90,width=0.6\textwidth]{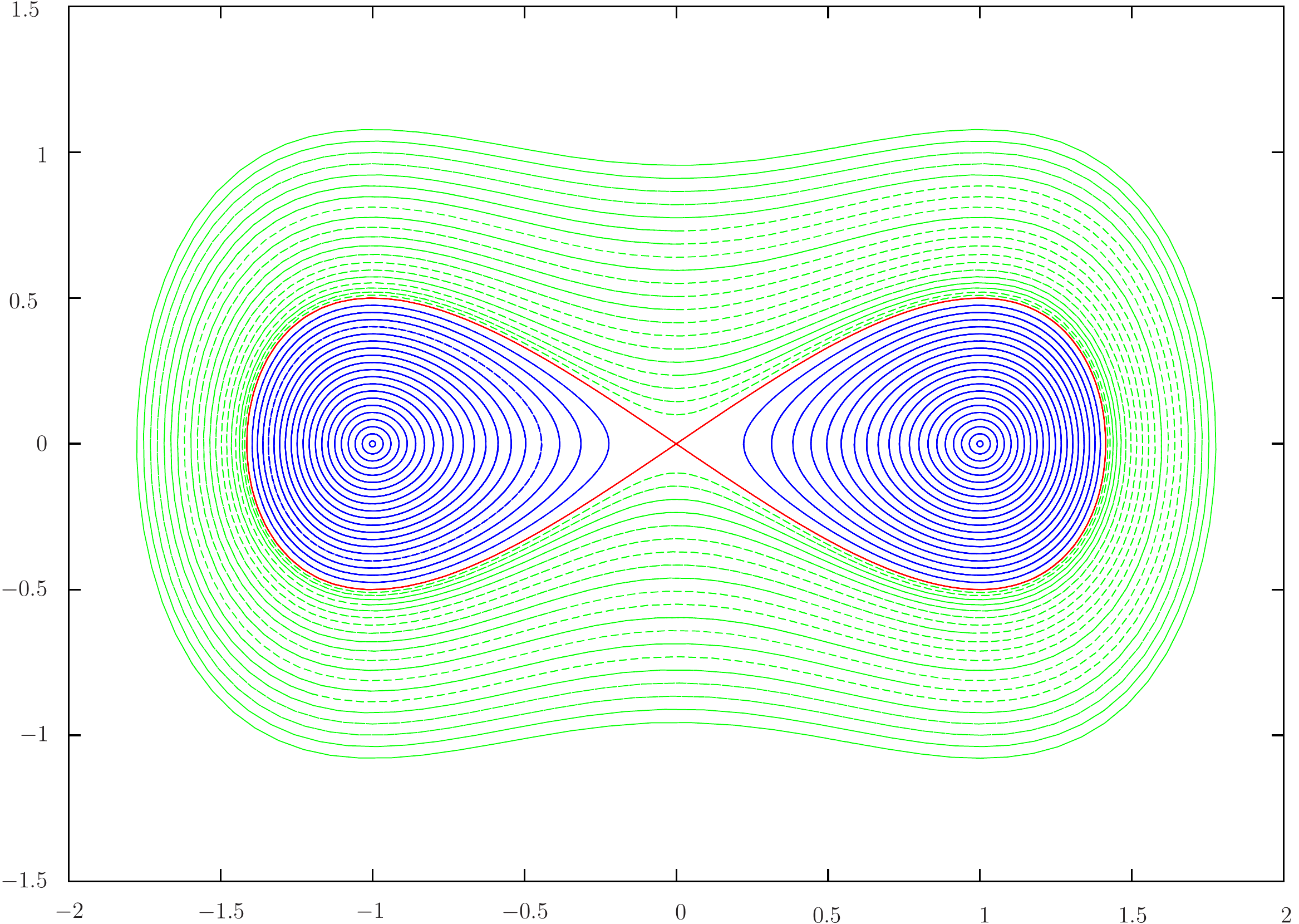}
\end{center}
\caption{Phase portrait of each beam in absence of forcing, damping
and coupling.}
\label{fig:figure_eight}
\end{figure}
It possesses three equilibrium points, two of the elliptic  type,
$Q^\pm=(\pm1,0)$, and a saddle point at the origin, $Q_0=(0,0)$. The
latter possesses two homoclinic loops,
$\gamma^\pm$, each surrounding the elliptic points $Q^\pm$,  and
located at the level $\left\{\X(x,y)=0\right\}$:
\iffalse
\begin{equation}
\gamma^\pm:=\W^s(Q_0)\cap \W^u(Q_0)\cap\left( \RR^\pm\times\RR \right),
\end{equation}
\begin{equation}
\gamma^+\cup\gamma^-=\W^s(Q_0)\cap \W^u(Q_0)=\W^s(Q_0)= \W^u(Q_0),
\end{equation}
such that
\fi%
\begin{equation*}
\gamma^+\cup\left\{ Q_0\right\}\cup \gamma^-=\left\{ (x,y)\in\RR^2,\,\X(x,y)=0
\right\}.
\end{equation*}
Therefore, the Hamiltonian $\X$ satisfies condition {\em h1} where the
homoclinic loop $\gamma$ can be either $\gamma^+$ or $\gamma^-$, which
are parameterized by two different parameterizations $\sigma^\pm$
satisfying
\begin{equation}
\gamma^\pm=\left\{ \sigma^\pm(t),\,t\in\RR \right\}
\label{eq:param_homo_loop}
\end{equation}
\begin{align*}
&\lim_{t\to\pm\infty}\sigma^\pm(t)=Q_0\\
&\sigma^\pm(0)=(0,\pm\sqrt{2}).
\end{align*}
Regarding the Hamiltonian $\U$, there exist three regions where it is
integrable and satisfies condition {\em h2}. Two of these three
regions are the ones surrounded by the homoclinic loops $\gamma^\pm$
and containing the elliptic points $Q^\pm$, and satisfy $\left\{
-\frac{1}{8}<U(u,v)<0 \right\}$. The third region consists of the
outer part to homoclinic loops, given by $\left\{ U(u,v)>0
\right\}$. These three regions are covered by a continum of periodic
orbits with growing period as approaching the homoclinic loop, hence
satisfying condition {\em h2}.\\
From the applied point of view, we are interested on having as much
energy as possible. Therefore, we will focus on the higher energy
periodic orbits located in this third region, as they provide larger
amplitude oscillations. Similar invariant objects to the ones we will
construct in Section~\ref{sec:unperturbed_objects} can be obtained
when focusing on the other two regions surrounding each of the
elliptic points $Q^\pm$.

For $\varepsilon=0$, the parameterizations given in 
Lemma~\ref{lem:NHIM_general} become as follows. We first fix the
transversal section to the periodic orbits as $u=0$. Hence, we get
$(u_c,v_c)=(0,\sqrt{2c})$ and the action-angle-like parameterization
of the periodic orbits becomes
\begin{equation}
\begin{array}{cccc}
\p:&\T\times \RR&\longrightarrow&\RR^2\\
&(\theta,c)&\longmapsto&\varphi_\U(\theta T_c;0,\sqrt{2c}).
\end{array}
\label{eq:change_variables_piezo}
\end{equation}
Provided the form of the Hamiltonian $\U$, we can obtain expressions
for the periods of $\Po^c$ as follows. The periodic orbit with initial
condition $(0,\sqrt{2c})$ crosses the $u$ axis at the point
\begin{equation*}
\left( u_1, 0 \right)=\left( \sqrt{1+\sqrt{1+8c}},0 \right).
\end{equation*}
Using the symmetries of the system and its Hamiltonian structure, we
obtain that  the period of the periodic orbit $\Po^c$ becomes
\begin{equation}
\begin{aligned}
T_c:=&4 \int_0^{u_1}\frac{1}{\dot{u}}du\\
=&4\int_0^{u_1}\frac{1}{\sqrt{2c+\frac{u^2}{2}\left(
1-\frac{u^2}{2} \right)}}du.
\end{aligned}
\label{eq:periodsU}
\end{equation}
However, as it will be detailed in
Section~\ref{sec:unperturbed_bundles}, when numerically computed, it
becomes better to compute $T_c$ using a Newton method instead of
numerically computing the integral~\eqref{eq:periodsU}.\\
We also get more concrete expressions for the parameterization of
$\tL$ given in Lemma~\ref{lem:NHIM_general}. In partilar,
$w_0^p$ and $w_0^u$ become
\begin{align}
w^p_0(u,v)&=-\frac{e^{-\lambda
T_c}}{1-e^{-\lambda T_c}}\int_0^{T_c} \kappa
v^p(s;u,v)e^{\lambda s}ds\label{eq:w0p_piezo}\\
v^p(s;u,v)&=\Pi_v\left(\varphi_\U(s;u,v)\right),\nonumber
\end{align}
and 
\begin{align}
w^u_0(\tau,u,v)&=-\int_{-\infty}^0\kappa e^{\lambda s}y^h(s;\tau)ds +
w^p_0(u,v)
\label{eq:condini_wu_piezo}\\
y^h(s;\tau)&=\Pi_y\left(\sigma(\tau+s)\right)\nonumber,
\end{align}
where $\sigma$ can be either $\sigma^+$ or $\sigma^-$.

\section{Numerical framework for the parameterization method}\label{sec:parametrization_method}
In this section we present a method to numerically compute the
perturbed Normally Hyperbolic Manifold $\tL_\varepsilon$ by means of
the so-called parameterization method, introduced
in~\cite{CabFondlL03,CabFondlL03b,CabFondlL05}. This method is the one
presented in~\cite{HarCanFigLuqMon16}, Chapter~5.\\
The parameterization method is stated easier when formulated for maps;
hence, it will be more convenient to write the full
system~\eqref{eq:pert_syst} as discrete system. The most natural map
to consider is of course the stroboscopic map, due to the periodicity
of the system. However, due to the foliated nature of the unperturbed
manifold $\tL_0$, instead of using the original variables, we will
consider from now on the induced map in the manifolds introduced in
Section~\ref{sec:unper_syst}:
\begin{equation}
\begin{array}[]{cccc}
F_\varepsilon:&\T\times(0,\infty)\times\RR^3&\longrightarrow&\T\times(0,\infty)\times\RR^3\\
&(\theta,c,x,y,w)&\longmapsto&\tp\circ\stro_\varepsilon\circ\tp^{-1}(\theta,c,x,y,w),
\end{array}
\label{eq:full_map}
\end{equation}
where $\tp$ is the extended version of  the change of
variables~\eqref{eq:change_variables_piezo},
\begin{equation*}
\tp(\theta,c,x,y,w)=(x,y,p(\theta,c),w)=(x,y,\varphi_\U(\theta
T_c;u_c,v_c),w),
\end{equation*}
and $\stro_\varepsilon$ is the stroboscopic map
\begin{equation*}
\begin{array}{cccc}
\stro_\varepsilon:&\Sigma_{s}&\longrightarrow&\Sigma_{s+T}\\
&(x,y,u,v,w)&\longmapsto&\tphi_\varepsilon(s+T;x,y,u,v,w,s),
\end{array}
\end{equation*}
being $\tphi_\varepsilon$ the flow associated with the full
system~\eqref{eq:pert_syst}.\\
Arguing as in Section~\ref{sec:persistence}, by considering the flow
$g'$ we obtain a map $F'_\varepsilon$ and a
parameterization\footnote{The absence of tilde in the objects indicates
that time coordinate has been dropped and that these objects refer
to the discrete system.}
\begin{equation}
K_\varepsilon:\T\times[c_1',c_2']\longrightarrow\T\times
[c_1',c_2']\times\RR^3,
\label{eq:param_map}
\end{equation}
such that the manifold
\begin{equation*}
\K'_\varepsilon=K_\varepsilon(\T\times[c_1',c_2'])
\end{equation*}
is unique, normally hyperbolic and invariant under $F_\varepsilon'$.
Moreover, $F_\varepsilon'$ and $F_\varepsilon$ coincide in
$\K'_\varepsilon$.  By restricting the parameterization
$K_\varepsilon$ to $[c_1,c_2]$, we obtain a manifold,
\begin{equation*}
\K_\varepsilon=K_\varepsilon(\T,[c_1,c_2]),
\end{equation*}
which is unique and normally hyperbolic. Although it is locally
invariant, $\K_\varepsilon$ contains the image of those points
isolated enough from the boundaries of $\K_0$ ($c=c_1$ and
$c=c_2$).\\
The map $F_\varepsilon$ restricted to $K_\varepsilon$ induces inner
dynamics in $\K_\varepsilon$ given by a map
\begin{equation*}
f_\varepsilon:\T\times[c_1,c_2]\longrightarrow\T\times[c_1',c_2'].
\end{equation*}
We find the inner dynamics $f_\varepsilon$ and the parameterization
$K_\varepsilon$, using the parameterization, that is, by imposing that
the diagram 
\begin{equation*}
\begin{CD}
\T\times[c_1,c_2]@>K_\varepsilon>> \T\times[c_1,c_2]\times\RR^3\\
@VVf_\varepsilon V @VVF_\varepsilon V\\
\T\times[c_1',c_2']@>K_\varepsilon>> \T\times[c_1',c_2']\times\RR^3
\end{CD}
\end{equation*}
commutes.\\
Note that, although we do not have an explicit expression for the map
$F_\varepsilon$, we can consider that it is given provided that we can
numerically compute the stroboscopic map $\stro$ just by integrating
the system. Hence, we need to solve the cohomological
\begin{equation}
F_\varepsilon\circ K_\varepsilon-K_\varepsilon\circ f_\varepsilon=0
\label{eq:cohomo}
\end{equation}
for the unknowns $f_\varepsilon$ and $K_\varepsilon$. We do this by
following  the method described in~\cite{HarCanFigLuqMon16} (Chapter
5). It consists of taken advantage of the hyperbolicity of
$\K_\varepsilon$ to perform a Newton-like method to compute functions
$f_\varepsilon$ and $K_\varepsilon$. In practice, given a
discretization of the reference manifold $\Tt\times [c_1,c_2]$, this
means that we want to compute the coefficients for approximations
(splines, Fourier series or Lagrangian polynomials) of $f_\varepsilon$
and $K_\varepsilon$. We first review the method described
in~\cite{HarCanFigLuqMon16} adapted to our case.
\subsection{A Newton-like method}\label{sec:newton_method}
Assume that, for a certain $\varepsilon>0$, we have good enough
approximations of $K_\varepsilon$ and $f_\varepsilon$. As usual in
Newton-like methods, in order to provide improved approximations, we will
require as well an initial guess of the dynamics
at the tangent bundles; that is, approximations of the solutions to
the cohomological equation
\begin{equation}
DF_\varepsilon\left( K_\varepsilon\left( \theta,c \right)
\right)-DK_\varepsilon(f_\varepsilon(\theta,c))=0.
\label{eq:cohomo_tangent}
\end{equation}
Note that Equation~\eqref{eq:cohomo_tangent} manifests the invariance
of the tangent bundle $T\K_\varepsilon$ under $DF_\varepsilon$,
leading to inner dynamics at $T\K_\varepsilon$ given by
$Df_\varepsilon$. However, $DF_\varepsilon$ is a map onto the tangent
space $T\left(\T\times\RR\times\RR^3\right)$. The latter can be
generated by two vectors in $T\K_\varepsilon$ and three in the normal
bundle $N\K_\varepsilon$, $T\left( \T\times\RR\times\RR^3
\right)=T\K_\varepsilon\times N\K_\varepsilon$. Provided that
$\K_\varepsilon$ is normally hyperbolic, the normal space
$N\K_\varepsilon$ can be generated by two vectors  tangent to
$\W^s(\K_\varepsilon)$ and a third one tangent to $\W^u(\K_ề)$. It
will be hence useful to consider the adapted frame around
$\K_\varepsilon$
\begin{equation*}
P_\varepsilon(\theta,c)=\left( L_\varepsilon(\theta,c)\,N_\varepsilon(\theta,c) \right),
\end{equation*}
given by the juxtaposition of the matrices $L_\varepsilon$ and
$N_\varepsilon$, where
\begin{equation*}
L_\varepsilon(\theta,c)=DK_\varepsilon(\theta,c)
\end{equation*}
is a $5\times2$ matrix and $N_\varepsilon(\theta,c)$ is $5\times3$ given by the three
columns of $DF_\varepsilon$ generating the normal bundle
$N\K_\varepsilon$.\\
The matrix $P_\varepsilon$ can be seen as a change of basis such that
the skew product
\begin{equation*}
(f_\varepsilon,\Lambda_\varepsilon):\T\times[c_1,c_2]
\times\RR^5
\longrightarrow \T\times[c_1,c_2]\times\RR^5,
\end{equation*}
with
\begin{equation*}
\Lambda_\varepsilon=P_\varepsilon(f(\theta,c))^{-1}DF_\varepsilon(K_\varepsilon(\theta,c))P_\varepsilon(\theta,c),
\end{equation*}
becomes of the form
\begin{equation*}
\Lambda_\varepsilon=\left( 
\begin{array}[]{ccccc}
%\multicolumn{2}{c}{\multirow{2}{*}{$(\Lambda_L)$}\rdelim{)_{2\times2}}{2}{10pt}}&0&0&0\\
\multicolumn{2}{c}{\multirow{2}{*}{$\Lambda^L_\varepsilon$}}&0&0&0\\
&&0&0&0\\
0&0&\multicolumn{2}{c}{\multirow{2}{*}{$\Lambda^S_\varepsilon$}}&0\\
0&0&&&0\\
0&0&0&0&\Lambda^U_\varepsilon
\end{array}
 \right).
\end{equation*}
Note that
\begin{equation*}
\Lambda^L_\varepsilon=Df_\varepsilon.
\end{equation*}
The Newton-like method consists of two steps. Given approximations of
$K_\varepsilon$ (and hence $L_\varepsilon=DK_\varepsilon$),
$f_\varepsilon$ (and hence $\Lambda^L_\varepsilon=Df_\varepsilon$),
$\Lambda^S_\varepsilon$ and $\Lambda^N_\varepsilon$, in the first
step, one computes new corrected versions of $K_\varepsilon$ and
$f_\varepsilon$ (and hence corrected versions of $L_\varepsilon$ and
$\Lambda^L_\varepsilon$). In the second step, one corrects the normal
bundle $N_\varepsilon$ and its linearized dynamics
$\Lambda^N_\varepsilon$.
\subsubsection{First step: correction of the manifold and the inner dynamics}\label{sec:step1}
As in~\cite{HarCanFigLuqMon16}, we consider corrections of the form
\begin{align*}
\bar{f}_\varepsilon&=f_\varepsilon+\Delta f_\varepsilon\\
\bar{K}_\varepsilon&=K_\varepsilon+\Delta K_\varepsilon,
\end{align*}
with\footnote{We permit ourselves here to keep the same notation as in
the literature and call this correction $\zeta$. Although this
coincides with the with the damping parameter from the original
system~\eqref{eq:model_spring}, we believe that it will be
clear from the context what we are referring to.}
\begin{equation*}
\Delta K_\varepsilon=P_\varepsilon(\theta,c) \zeta(\theta,c).
\end{equation*}
We want to compute $\zeta(\theta,c)$ and $\Delta
f_\varepsilon(\theta,c)$.\\
The corrections of the manifold, $\zeta(\theta,c)$, can be split in
tangent, stable and unstable directions:
\begin{equation*}
\zeta(\theta,c)=\left( 
\begin{array}[]{c}
\zeta^L(\theta,c)\\\zeta^S(\theta,c)\\\zeta^U(\theta,c)
\end{array}
 \right)\in \RR^2\times\RR^2\times\RR.
\end{equation*}
Let $E(\theta,c)$ be the error at Equation~\eqref{eq:cohomo}
at the current value $f_\varepsilon$ and $K_\varepsilon$,
\begin{equation*}
E(\theta,c)=F_\varepsilon(K_\varepsilon(\theta,c))-K_\varepsilon(f_\varepsilon(\theta,c)).
\end{equation*}
Let
\begin{equation*}
\eta(\theta,c)=-\left( P(f(\theta,c)) \right)^{-1}E(\theta,c),
\end{equation*}
which, again, we split in tangent, stable and unstable directions
\begin{equation*}
\eta(\theta,c)=\left( 
\begin{array}[]{c}
\eta^L(\theta,c)\\\eta^S(\theta,c)\\\eta^U(\theta,c)
\end{array}
 \right).
\end{equation*}
Then, assuming that
\begin{equation*}
\zeta^L(\theta,c)=\left( 
\begin{array}[]{c}
0\\0
\end{array}
 \right),
\end{equation*}
that is, the manifold $K_\varepsilon(\T,[c_1,c_2])$ is only corrected
in the normal directions, $\Delta f_\varepsilon$ becomes
\begin{equation*}
\Delta f_\varepsilon(\theta,c)=-\eta^L(\theta,c)
\end{equation*}
(see~\cite{HarCanFigLuqMon16} for more details).\\
Regarding $\Delta K_\varepsilon$, by expanding
Equation~\eqref{eq:cohomo}, one gets that $\zeta^S$ and $\zeta^U$
satisfy the equations
\begin{align}
\eta^S(\theta,c)&=\Lambda^S_\varepsilon(\theta,c)\zeta^S(\theta,c)-\zeta^S(\theta,c)\label{eq:zetaS}\\
\eta^U(\theta,c)&=\Lambda^U_\varepsilon(\theta,c)\zeta^U(\theta,c)-\zeta^U(\theta,c)\label{eq:zetaU}.
\end{align}
Due to hyperbolicity, $\Lambda^U_\varepsilon>1$ and the eigenvalues of
$\Lambda^S_\varepsilon$ are real, positive and inside the unite
circle. Hence, Equations~\eqref{eq:zetaS}-\eqref{eq:zetaU} can be
solved by iterating the systems
\begin{align*}
\zeta^S(\theta,c)&=\Lambda^S_\varepsilon(f^{-1}(\theta,c))\zeta^S(f^{-1}(\theta,c))-\eta^S(f^{-1}(\theta,c))\\
\zeta^U(\theta,c)&=\frac{\zeta^U(f(\theta,c))+\eta^U(\theta,c)}{\Lambda^U_\varepsilon(\theta,c)}.
\end{align*}
\subsubsection{Second step: normal bundle correction}\label{sec:step2}
Once we have new corrected versions of the inner dynamics
$f_\varepsilon(\theta,c)$ (and consequently
$\Lambda^L_\varepsilon(\theta,c)$) and the parametrization
$K_\varepsilon(\theta,c)$ (and consequently $L_\varepsilon(\theta,c)$), the second
step consists of computing new approximations for
$\Lambda^S_\varepsilon$, $\Lambda^U_\varepsilon$ and
$N_\varepsilon(\theta,c)$. As in~\cite{HarCanFigLuqMon16}, we consider
approximations of the form
\begin{align*}
\bar{N}_\varepsilon(\theta,c)&=N_\varepsilon(\theta,c)+\Delta
N(\theta,c)\\
\bar{\Lambda}^S_\varepsilon(\theta,c)&=\Lambda^S_\varepsilon(\theta,c)+\Delta\Lambda^S(\theta,c)\\
\bar{\Lambda}^U_\varepsilon(\theta,c)&=\Lambda^U_\varepsilon(\theta,c)+\Delta\Lambda^U(\theta,c),
\end{align*}
with
\begin{equation*}
\Delta N(\theta,c)=P_\varepsilon(\theta,c)Q^N(\theta,c).
\end{equation*}
Assuming that the corrections of the linearised stable and unstable
bundles are applied only in the complementary directions, one gets
that $Q^N(\theta,c)$ is of the form
\begin{equation*}
Q^N(\theta,c)=
\left( 
\begin{array}[]{ccc}
\multicolumn{2}{c}{\multirow{2}{*}{$Q^{LS}(\theta,c)$}}&\multirow{2}{*}{$Q^{LU}(\theta,c)$}\\
&&\\
0&0&\multirow{2}{*}{$Q^{SU}(\theta,c)$}\\
0&0&\\
\multicolumn{2}{c}{Q^{US}(\theta,c)}&0
\end{array}
 \right).
\end{equation*}
Let us write the current error in the cohomological
equation at the tangent bundle (Equation~\eqref{eq:cohomo_tangent}) as
\begin{align*}
\Ered^N(\theta,c)&=P_\varepsilon(\theta,c)^{-1}DF_\varepsilon\left(
K_\varepsilon(\theta,c) \right)N_\varepsilon(\theta,c)-\\
&-\left( 
\begin{array}[]{ccc}
0&0&0\\
0&0&0\\
\multicolumn{2}{c}{\multirow{2}{*}{$\Lambda_\varepsilon^S(\theta,c)$}}&0\\
&&0\\
0&0&\Lambda^U_\varepsilon(\theta,c)
\end{array}
\right)=\left( 
\begin{array}[]{ccc}
\multicolumn{2}{c}{\multirow{2}{*}{$\Ered^{LS}(\theta,c)$}}&\multirow{2}{*}{$\Ered^{LU}(\theta,c)$}\\
&&\\
\multicolumn{2}{c}{\multirow{2}{*}{$\Ered^{SS}(\theta,c)$}}&\multirow{2}{*}{$\Ered^{SU}(\theta,c)$}\\
&&\\
\multicolumn{2}{c}{\Ered^{US}(\theta,c)}&\Ered^{UU}(\theta,c)
\end{array}
 \right)
\end{align*}

On one hand, the corrections of the adapted frame in the normal
directions become
\begin{align*}
\Delta\Lambda^S(\theta,c)&=\Ered^{SS}(\theta,c)\\
\Delta\Lambda^U(\theta,c)&=\Ered^{UU}(\theta,c).
\end{align*}
On the other hand, the corrections of normal part of the change of basis
$P_\varepsilon(\theta,c)$, are obtained by iterating the systems
\begin{align*}
Q^{LS}(\theta,c)&=\left( \Lambda^L_\varepsilon(\theta,c) \right)\left(
Q^{LS}(f(\theta,c))\Lambda^S_\varepsilon(\theta,c)-\Ered^{LS}(\theta,c)
\right)\\
Q^{LU}(\theta,c)&=\frac{\Lambda_\varepsilon^L(f^{-1}(\theta,c))+\Ered^{LU}(f^{-1}(\theta,c))}{\Lambda^U_\varepsilon(f^{-1}(\theta,c))}\\
Q^{US}(\theta,c)&=\frac{Q^{US}(f(\theta,c))\Lambda^S_\varepsilon(\theta,c)-\Ered^{US}(\theta,c)}{\Lambda^U_\varepsilon(\theta,c)}\\
Q^{SU}(\theta,c)&=\frac{\Lambda^S_\varepsilon(f^{-1}(\theta,c))Q^{SU}(f^{-1}(\theta,c))+\Ered^{SU}(f^{-1}(\theta,c))}{\Lambda^U_\varepsilon(f^{-1}(\theta,c))}.
\end{align*}
\subsection{Computation of bundles, maps and frames for the unperturbed piezoelectric system}\label{sec:unperturbed_bundles}
We now derive semi-explicit expressions for the objects for
$\varepsilon=0$ for system~\eqref{eq:pert_syst} with $g_i$ given in
Equations~\eqref{eq:g0_piezo}-\eqref{eq:g2_piezo}: $f_0(\theta,c)$,
$F_0(\theta,c,x,y,w)$,\linebreak $DF_0(\theta,c)$ and $P_0(\theta,c)$. The
expressions below have to be partially solved numerically. By
semi-explicit we mean that we will assume that computations such us
numerical integration or differentiation is exact.

For $\varepsilon=0$, the map $F_0(\theta,c,x,y,w)$ consists of
computing the stroboscopic map, integrating
system~\eqref{eq:unpert_syst}, from $t_0$ to $t_0+T$.  However, we
first need to compute the change of variables $p(\theta,c)=(u,v)$,
which requires the computation of $T_c$.  The expression in
Equation~\eqref{eq:periodsU} implies computing an improper integral
provided that $\dot{u}=0$ at $u=u_1$, which is numerically
problematic. Instead, we perform a Newton method to find the smallest
$t^*>0$ such that
\begin{equation*}
\Pi_u\left(\varphi_\U(t^*;0,\sqrt{2c})\right)=0.
\end{equation*}
Then $T_c=2t^*$ due to the symmetry of the system. This is given
by iterating the system
\begin{equation*}
t_{i+1}=t_i -\frac{\Pi_u\left( \varphi_\U(t_i;0,\sqrt{2c})
\right)}{\Pi_v\left( \varphi_\U(t_i;0,\sqrt{2c}) \right)},
\end{equation*}
which allows to obtain a very accurate solution (precision around
$10^{-12}$ using order $7--8th$ order Runge Kutta integrator) in few
iterations assuming that a good enough first guess is provided.\\
For $\varepsilon=0$, the system is autonomous and the stroboscopic map
does not depend on $t_0$; it becomes
\begin{equation*}
\stro_0(x,y,p(\theta,c),w)=\left( 
\begin{array}[]{c}
\varphi_\X(T;x,y)\\\varphi_\U(\theta T_c+T;0,\sqrt{2c})\\w(T;w)
\end{array}
 \right),
\end{equation*}
where $w(t;w)$ is the solution of equation
\begin{equation}
\dot{w}=-\lambda w-k(\Pi_y\left( \varphi_\X(t;x,y)
\right)+\Pi_v(\varphi_\U(\tau;p(\theta,c)))).
\label{eq:w-equation}
\end{equation}
Recall that, as mentioned in Section~\ref{sec:unper_syst},
Equation~\eqref{eq:w-equation} can be seen as a one-dimensional
non-autonomous differential equation, assuming that the flows
$\varphi_\X$ and $\varphi_\U$ are known, but indeed depends on the
initial values $x,y,u,v$.\\
Writing $\stro_0$ in variables $\theta,c$, we obtain the map $F_0$
\begin{equation*}
F_0(\theta,c,x,y,w)=\left( 
\begin{array}[]{c}
\theta+\frac{T}{T_c}\\c\\\varphi_\X(T;x,y)\\w(T;w)
\end{array}
 \right),
\end{equation*}
and we assume that we can integrate system $\X$ and
Equation~\eqref{eq:w-equation} ``exactly''.\\
For $\varepsilon=0$ the parameterization $K_0(\theta,c)$ becomes
\begin{equation*}
K_0(\theta,c)=\left( 
\begin{array}[]{c}
\theta\\c\\0\\0\\w^p_0(p(\theta,c))
\end{array}
 \right).
\end{equation*}
Again, we need to compute $w^p_0(p(\theta,c))$ numerically. One option
is to numerically perform the integral given in
Equation~\eqref{eq:w0p}.  However, it becomes faster an more precise
to compute $w^p_0(p(\theta,c))$ as the solution of a fixed point
equation. Recall that, when restricted to the manifold $\tK_0$,
$(x,y)$ are kept constant to $(0,0)$, and hence the dynamics is given
by system $\U$ and $\dot{w}$ as given in Equation~\eqref{eq:bp}, which
we recall here for commodity 
\begin{equation}
\left.
\begin{aligned}
\dot{u}&=v\\
\dot{v}&=\frac{1}{2}u \left( 1-u^2 \right)\\
\dot{w}&=\lambda w -k v
\end{aligned}
\right\}.
\label{eq:system_uvw}
\end{equation}
Let $\varphi_{\U w}(t;u,v,w)$ be the flow associated with
system~\eqref{eq:system_uvw}. Then, $w^p_0(p(\theta,c))$ is the
solution 
 for $w_0$ of the fixed point
equation
\begin{equation}
\Pi_w\left( \varphi_{\U w}(T_c;p(\theta,c),w_0) \right)-w_0=0,
\label{eq:fixed_point_w0}
\end{equation}
which we can solve using a Newton method.  Provided that $T_c$
does not depend on $w_0$, the derivative $\partial_{w_0} \Pi_w\left(
\varphi_{\U x}(\alpha;p(\theta,c),w_0) \right)$, necessary for the
Newton method, can be obtained by integrating the variational
equations of system~\eqref{eq:system_uvw} from $t=0$ to $t=T_c$.

We next get 
\begin{equation*}
L_0(\theta,c)=DK_0(\theta,c)=\left( 
\begin{array}[]{cc}
1&0\\ 0&1\\0&0\\0&0\\ \frac{\partial
w^p_0(\theta,c)}{\partial\theta}&\frac{\partial
w_0^p(\theta,c)}{\partial c}
\end{array}
 \right),
\end{equation*}
and we need to numerically compute the last row. Assuming that we have
obtained $w^p_0(p(\theta,c))$, this can be done by applying the
implicit function theorem to Equation~\eqref{eq:fixed_point_w0}, which
leads to
\begin{align*}
\frac{\partial w^p_0(p(\theta,c))}{\partial \theta}&=
-\frac{\overbrace{D_{uv}\Pi_w\left( \varphi_{\U w}(T_c;u,v,w_0^p(u,v))
\right)_{(u,v)=p(\theta,c)}}^{*}\cdot D_\theta
p(\theta,c)}{\underbrace{\partial_{w_0}\Pi_w\left( \varphi_{\U
w}(T_c;p(\theta,c),w_0) \right)_{w_0=w_0^p(\theta,c)}}_*-1}\\
\frac{\partial w^p_0(p(\theta,c))}{\partial c}&=
-\frac{\overbrace{\Pi_w\left( \varphi_{\U
w}'(T_c;p(\theta,c),w^p_0(p(\theta,c)))
\right)}^{-\lambda w^p_0(p(\theta,c))-k \Pi_v(p(\theta,c))}\cdot\alpha'(c)}{\partial _{w_0}\Pi_w\left(\varphi_{\U w}\left(
T_c;p(\theta,c),w_0
\right)  \right)_{w_0=w_0^p(\theta,c)}-1}-\\
&\quad -\frac{\overbrace{D_{uv}\varphi_{\U
w}(T_c;u,v,w^p_0(p(\theta,c)))_{(u,v)=p(\theta,c)}}^*\cdot
D_cp(\theta,c)}{\underbrace{\partial _{w_0}\Pi_w\left(\varphi_{\U w}\left(
T_c;p(\theta,c),w_0
\right)  \right)_{w_0=w_0^p(\theta,c)}}_{*}-1}.
\end{align*}
The terms labeled with $*$ can be obtained by integrating the
variational equations of system~\eqref{eq:system_uvw}, so we still
need to obtain $\alpha'(c)$ and $D_{\theta c}p(\theta,c)$. The former
one is computed by finite differences provided that we can accurately
compute $\alpha_{c+h}$ and $\alpha_{c-h}$ for a small enough $h>0$.
The latter becomes
\begin{align*}
Dp(\theta,c)&=D_{\theta,c}\varphi_\U(\theta T_c;0,\sqrt{2c})\\
&=D_{uv}\varphi_\U(\theta T_c;u,v)_{(u,v)=p(\theta,c)}\cdot \left( 
\begin{array}[]{cc}
0&0\\ 0& \frac{1}{\sqrt{2c}}
\end{array}
 \right),
\end{align*}
where $D_{uv}\varphi_\U(\theta T_c;u,v)$ can be obtained
integrating the variational equations associated with system $\U$.

We now compute
\begin{equation*}
DF_0(\theta,c,x,y,w)=\left( 
\begin{array}[]{ccccc}
-\frac{T}{T_c^2}\alpha'(c) & 1 &0&0&0\\
0&1&0&0&0\\
0&0&\multicolumn{2}{c}{\multirow{2}{*}{$D_{xy}\varphi_\X(T;x,y)$}}&0\\
0&0&&&0\\
\partial_{\theta}w'&\partial_{c}w'&\partial_{x}w'
&\partial_{y}w'&\partial_{w}w'
\end{array}
 \right),
\end{equation*}
where $w'=w(2\pi/\omega;w)$ is the solution of
Equation~\eqref{eq:w-equation} which, as emphasized above, depends
also on $\theta,c,x,y$. The first element of $DF_0$ requires computing
$\alpha'(c)$, which we have already seen. The rest of the elements of
$DF_0$ can be computed  by integrating the full unperturbed
system~\eqref{eq:unpert_syst} together wit its variational equations.

When evaluated at the manifold  $\K_0$, the eigenvectors of $DF_0$
provide proper directions to split the normal space in stable and
unstable directions and hence to obtain the matrix $N_0(\theta,c)$.\\
As shown in Lemma~\eqref{lem:NHIM_general}, at each point of the manifold
$\tK_0$ (similarly for $\K_0$)  the normal bundle is split in two
stable directions and an unstable one. One stable direction is given
by the contraction in $w$; the other two are given by the stable and
unstable directions of the saddle point $Q_0$ of system $\X$. These
directions are given by the eigenvectors of matrix
$DF_0(\tK_0)$
\begin{align*}
v_1^s&=\left( 0,0,1,-1/\sqrt{2},P_{5,3} \right)\\
v_2^s&=\left( 0,0,0,0,1 \right)\\
v^u&=\left( 0,0,1,1/\sqrt{2},P_{5,5} \right),
\end{align*}
with
\begin{align*}
P_{5,3}&=\frac{1/\sqrt{2}\partial_yw'-\partial_xw'}{\partial_ww'+1/\sqrt{2}}\\
P_{5,5}&=-\frac{1/\sqrt{2}\partial_yw'+\partial_xw'}{\partial_ww'-1\sqrt{2}}.
\end{align*}
Hence, $N_0(\theta,c)$ becomes
\begin{equation*}
N_0(\theta,c)=\left( 
\begin{array}[]{ccc}
0&0&0\\0&0&0\\ 1&0&1\\
-1/\sqrt{2}&0&1/\sqrt{2}\\
P_{5,3}&1&P_{5,5}
\end{array}
 \right).
\end{equation*}
\section{Numerical results}\label{sec:numerical_results}
We apply the method described in~\cite{HarCanFigLuqMon16} (summarized
in Sections~\ref{sec:newton_method}-\ref{sec:unperturbed_bundles}) to
the map given in Equation~\eqref{eq:full_map} and corresponding to the
stroboscopic map of System~\eqref{eq:pert_syst} with $\X$, $\U$, $h$,
$g$ and $b$ given in Equations~\eqref{eq:X_piezo}--\eqref{eq:b_piezo}.
We hence will obtain numerical computations of the discrete versions
of the Normally Hyperbolic Manifold $\tL_\varepsilon$ and its
associated invariant manifolds $\W^s(\tL_\varepsilon)$ and
$\W^u(\tL_\varepsilon)$. We fix the following parameter values
\begin{align*}
G(t)&=\sin(\omega t)&\omega&=2.1\\
%\tzeta&=0&\tc&=0.1\\
\lambda&=0.02&\tk=\kappa&=1
%\beta&=1&\varepsilon&=2\cdot10^{-3},
\end{align*}
and consider different situations regarding parameters $\varepsilon$,
$\tzeta$ and $\tc$.  For each of them, we use as seed for the Newton
method the unperturbed setting shown in
Section~\ref{sec:unperturbed_bundles} and use cubic spline
interpolations in a grid of points for
$(\theta,c)\in\T\times[0.1,1.4]$. At each Newton step, the two
substeps explained in Sections~\ref{sec:step1} and~\ref{sec:step2} are
performed for each point of the grid, which allows a natural
parallelization. This is done using {\em OpenMP} libraries and, ran in
a 8 cores node with multithreading (16 threads), each Newton step
takes around 2 minutes for a grid of $500\times 200$ points. The code
is available at
%\href{http://github.com/a-granados/parameterization_piezoelectric}{\url{https://github.com/a-granados/parameterization_piezoelectric}}
\begin{center}
\url{https://github.com/a-granados/nhim_parameterization}
\end{center}
\subsection{Conservative case}\label{sec:numerical_conservative}
We start by setting $\tzeta=\tc=0$ and $\varepsilon=6\cdot10^{-2}$.
In this case, only the conservative terms of the coupling between the
oscillators remain active, as $\tk=1$. The inner dynamics, restricted
to $\tL_\varepsilon$, is given by the one and a half degrees of
freedom Hamiltonian system
\begin{equation*}
\U(u,v)+\varepsilon
h\left(\Pi_{x,y,u,v,s}\left(\tK_\varepsilon\left(p^{-1}_\varepsilon(u,v),s\right)\right)\right),
\end{equation*}
 and therefore $f_\varepsilon$ becomes a symplectic
map.
\begin{figure}
\begin{center}
\includegraphics[angle=-90,width=0.8\textwidth]{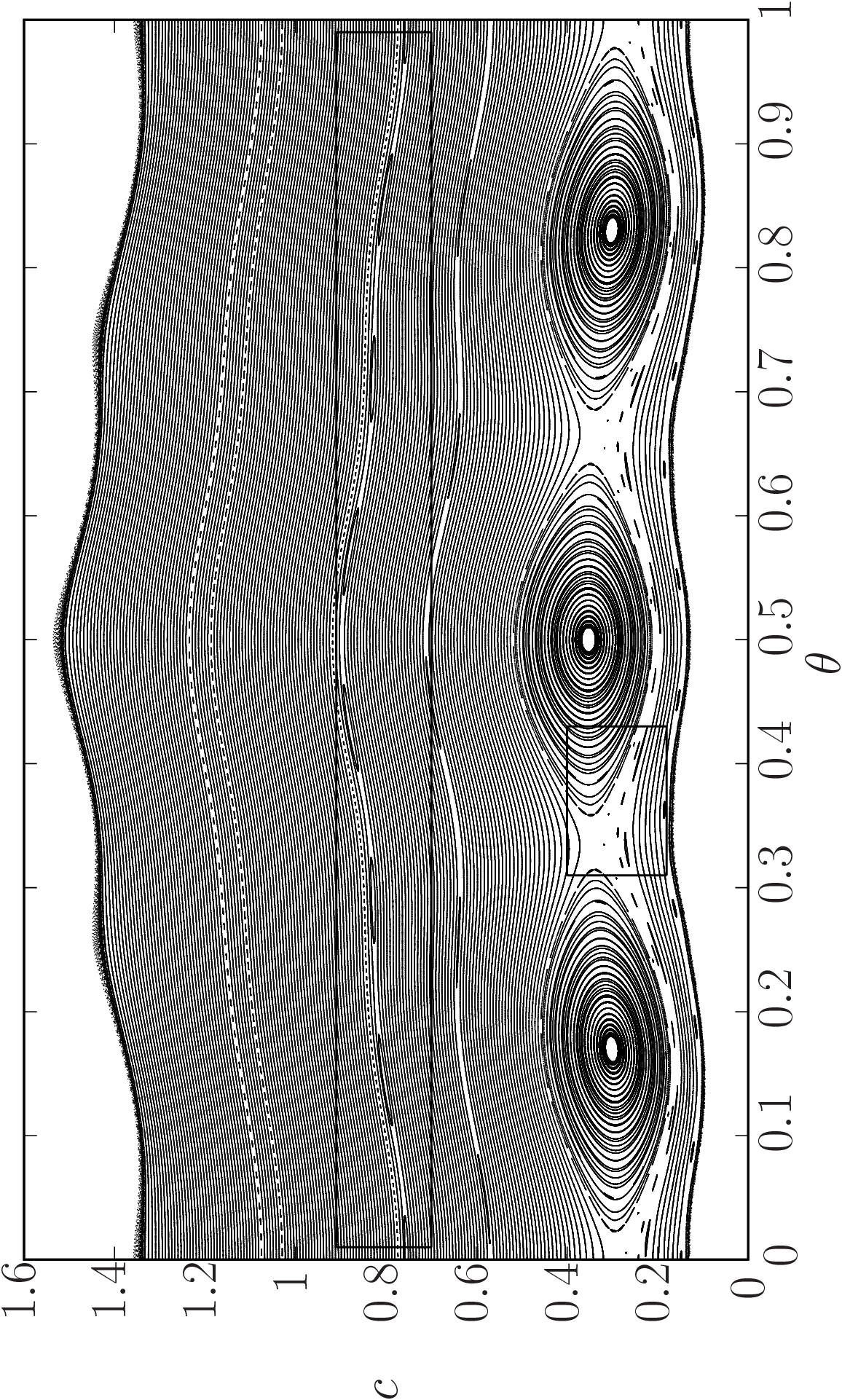}
\end{center}
\caption{Inner dynamics in $\K_\varepsilon$ for the conservative case
with $\varepsilon=6\cdot10^{-2}$ and $\tzeta=\tc=0$, obtained by iterating the
numerically obtained map $f_\varepsilon(\theta,c)$. The labeled
rectangles are magnified in Figures~\ref{fig:3:1_resonance}
and~\ref{fig:7:3_resonance}.}
\label{fig:conser_thetac}
\end{figure}
In Figure~\ref{fig:conser_thetac} we show the global picture of the
inner dynamics, where one can see typical objects of this type of
maps. The space is mostly covered by KAM invariant curves acting as
energy bounds. For the chosen value of $\omega$ one observes three
main resonances: $3:1$, $5:2$ and $7:3$, where $m:n$ means $mT=nT_c$.
\begin{figure}
\begin{center}
\includegraphics[angle=-90,width=0.8\textwidth]{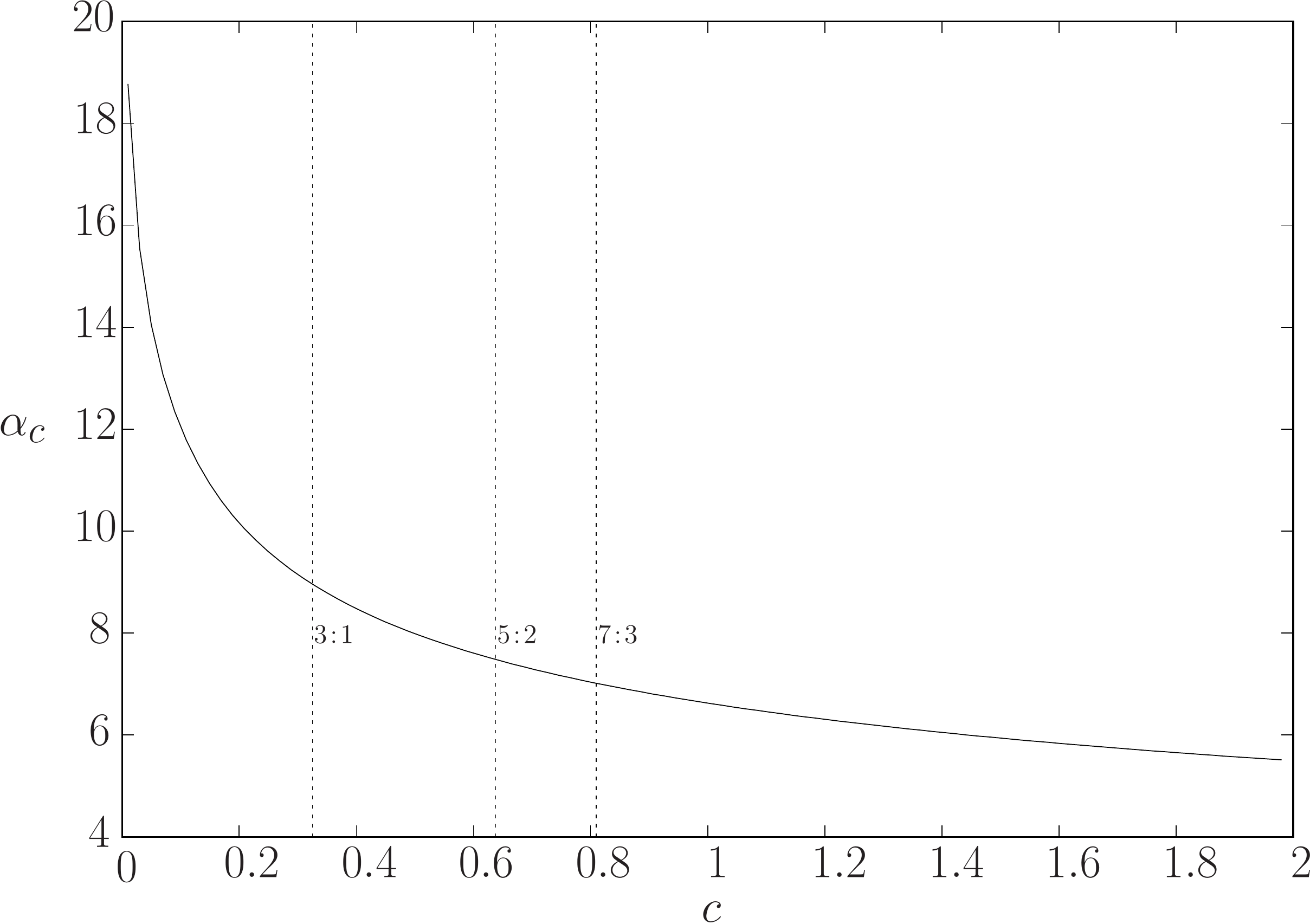}
\end{center}
\caption{Periods of the unperturbed system, $T_c$.}
\label{fig:unpert_periods}
\end{figure}
These resonances are labeled in Figure~\ref{fig:unpert_periods}, where
we show the periods ($T_c$) of the unperturbed system as a function of
$c$, and they approximately correspond to the unperturbed periodic
orbits $\Po_{0.325}$, $\Po_{0.638}$ and $\Po_{0.81}$ defined in
Equation~\eqref{eq:periodic_orbits}, respectively. As it comes from
Melnikov theory  (see~\cite{GucHol83}) for subharmonics orbits, when
they persist, one finds an even number of periodic orbits of the
stroboscopic map; half are of the saddle and the rest are elliptic;
this number is given by the number of simple zeros of the so-called
Melnikov function for subharmonic periodic orbits.  Recalling that we
are dealing with the conservative case ($\varepsilon>0$,
$\tzeta=\tc=0$), in our case this function becomes
\begin{equation}
M(t_0)=\int_0^{mT}v(t)\left(\tk u(t)+G(t+t_0)\right)dt,
\label{eq:melnikov_function}
\end{equation}
where $(u(t),v(t))=\varphi_\U(t;0,\sqrt{2c})$ is evaluated along the
unperturbed periodic orbit, $\Po^c$, with initial condition at $u=0$
and satisfying $T_c=mT/n$.
\begin{figure}
\begin{center}
\includegraphics[angle=-90,width=0.8\textwidth]{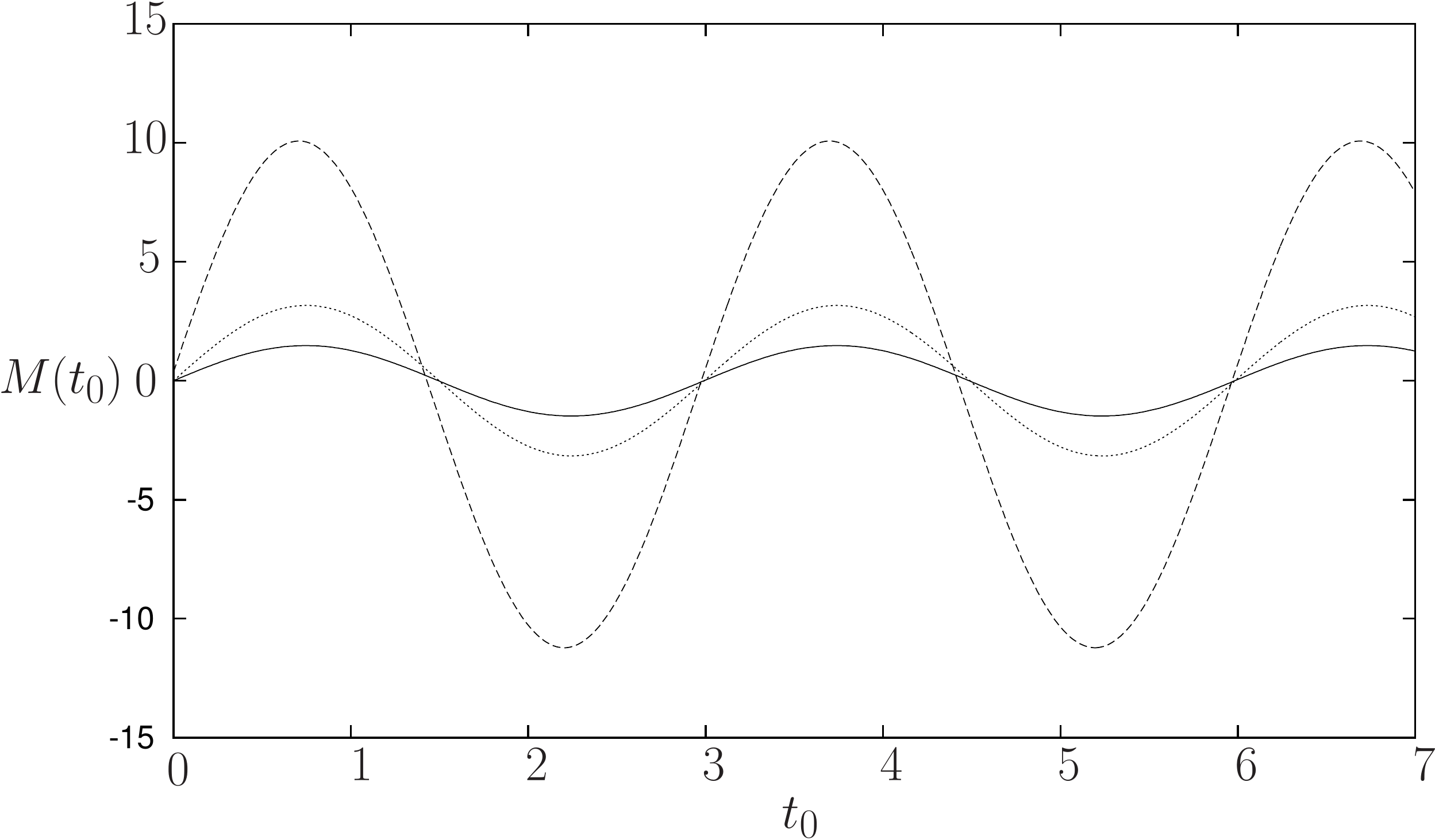}
\end{center}
\caption{Melnikov functions for the resonances $3:1$ (solid), $7:3$
(dashed, magnified by a factor of $10^2$) and $5:2$ (pointed, magnified
by a factor of $10^5$).}
\label{fig:melnikov_functions}
\end{figure}
In Figure~\ref{fig:melnikov_functions} we show such function for these
three resonances. Each of them possesses two simple zeros and, hence,
there exist (for $\varepsilon>0$ small enough) two periodic points of
the stroboscopic map of the saddle and elliptic type. By adding higher
harmonics to $G(t)$, function $M(t_0)$ may possess more simple zeros.
Note that all shown three functions have simple zeros at $t_0=0$,
which, recalling that the initial condition to compute $M(t_0)$ is
taken at $u=0$, implies that the corresponding periodic orbits possess
one point $\varepsilon$-close to $\theta=0$. Note also that the $7:3$
and $5:2$ Melnikov functions have been magnified by a factor $100$ and
$10000$, respectively, which tells us which of them will first
bifurcate when increasing~$\varepsilon$.\\
The $3:1$ periodic orbits are clearly observed in
Figure~\ref{fig:conser_thetac}.
\begin{figure}
\begin{center}
\includegraphics[angle=-90,width=0.8\textwidth]{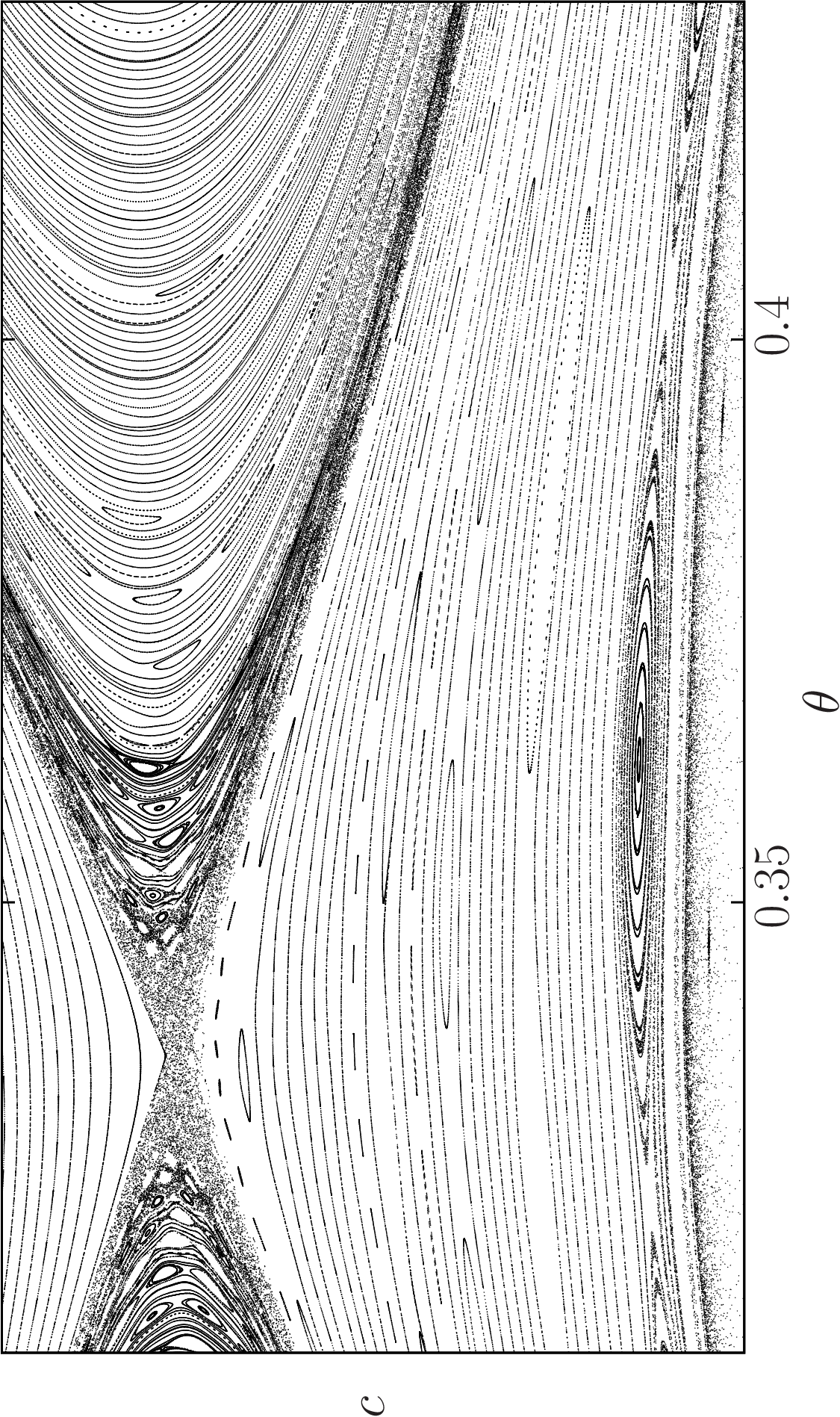}
\end{center}
\caption{Blow up of the $3:1$ resonance labeled in
Figure~\ref{fig:conser_thetac}}
\label{fig:3:1_resonance}
\end{figure}
The saddle type one is magnified in Figure~\ref{fig:3:1_resonance},
where one also observes secondary tori and evidence of chaos given by
the homoclinic tangles. The resonance $7:3$ is magnified in
Figure~\ref{fig:7:3_resonance}.
\begin{figure}
\begin{center}
\includegraphics[angle=-90,width=0.8\textwidth]{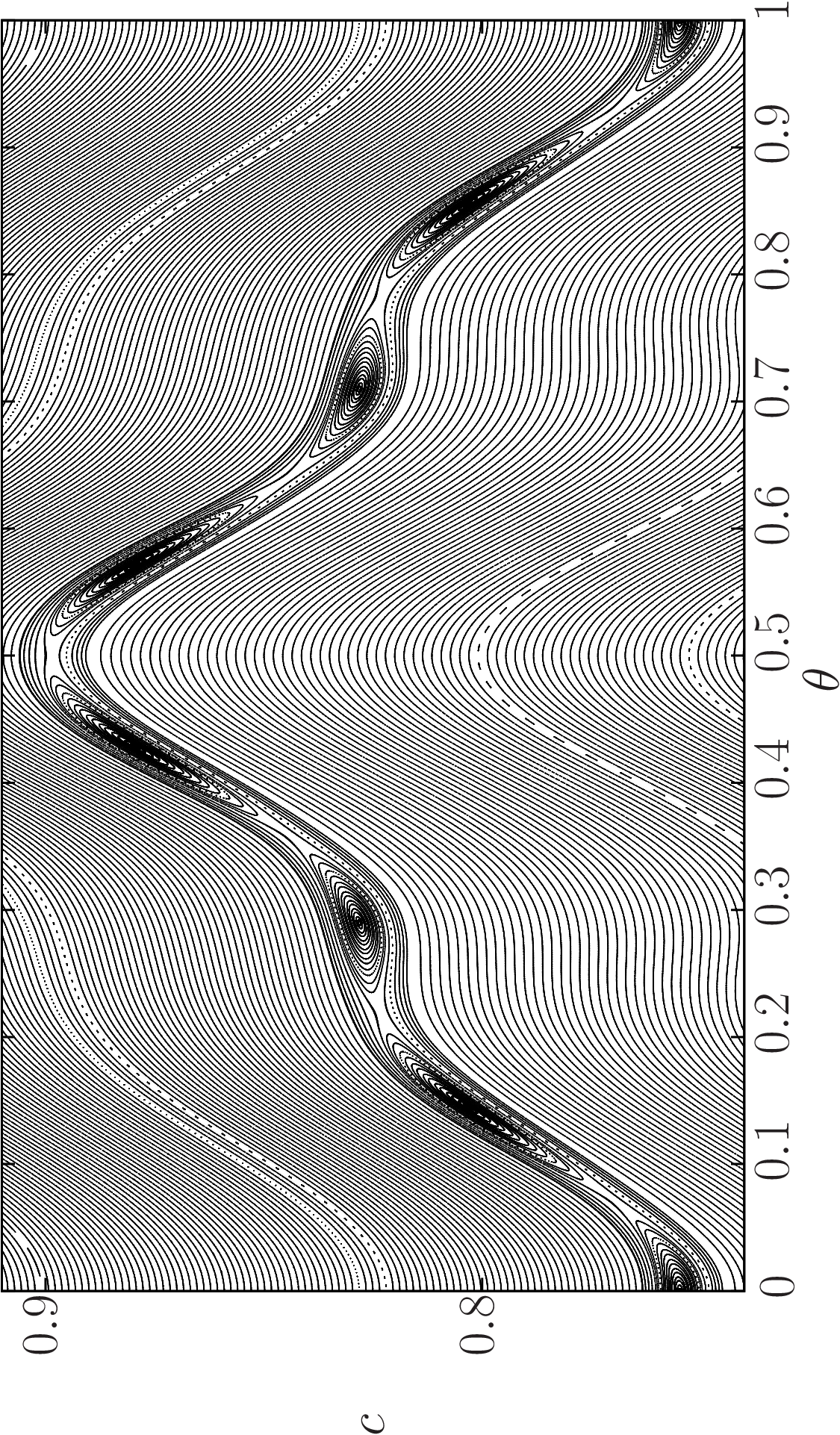}
\end{center}
\caption{Blow up of the $7:3$ resonance labeled in
Figure~\ref{fig:conser_thetac}}
\label{fig:7:3_resonance}
\end{figure}
\begin{figure}
\begin{center}
\begin{picture}(1,0.5)
\put(0,0.25){
%\subfigure[\label{fig:conser_thetacw}]{\includegraphics[angle=-90,width=0.5\textwidth]{eps5e-2_thetacw/inner_dynamics}}
\subfigure[\label{fig:conser_thetacw}]{\includegraphics[angle=-90,width=0.5\textwidth]{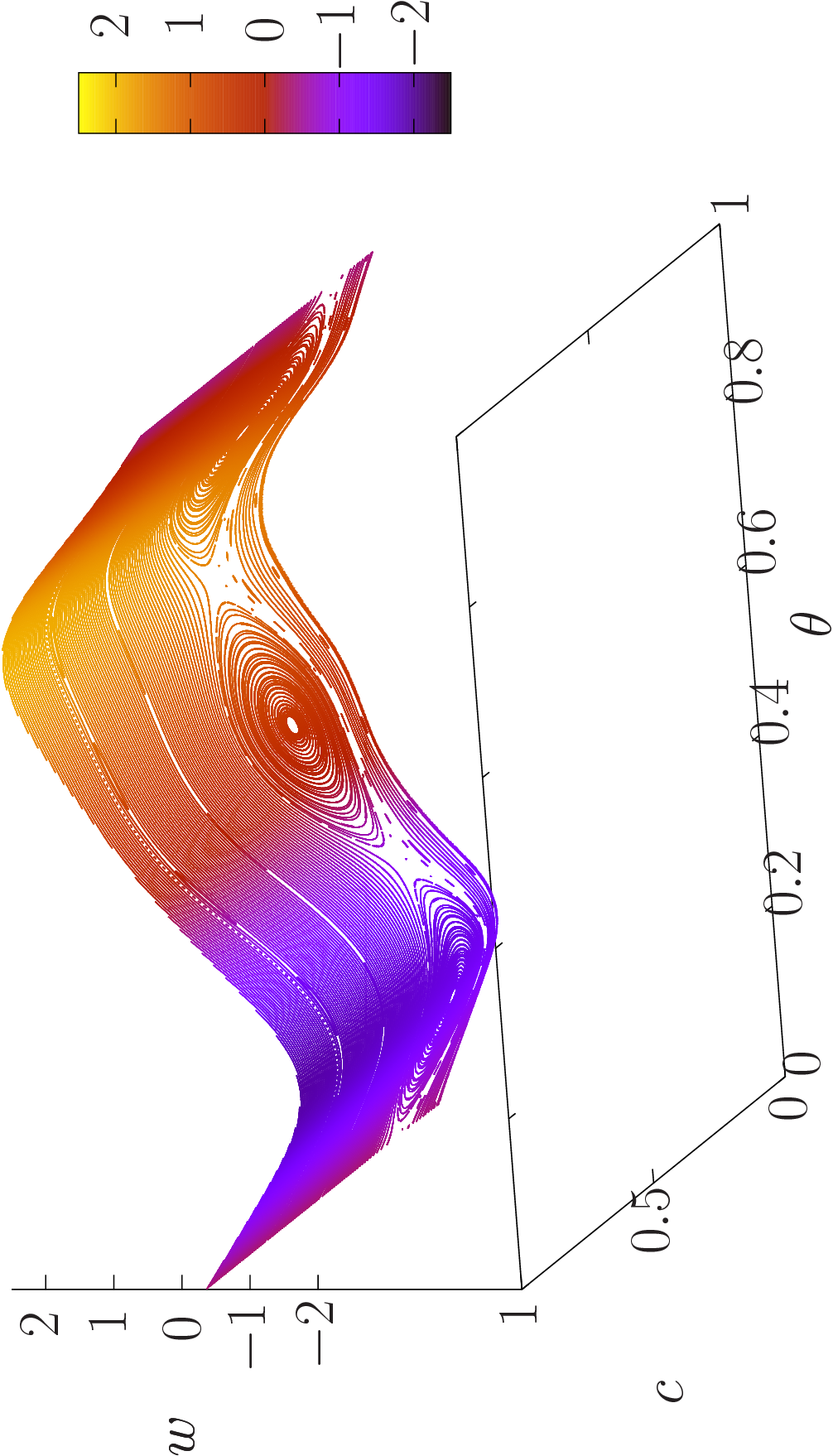}}
}
\put(0.51,0.23){
%\subfigure[\label{fig:conser_xyc}]{\includegraphics[angle=-90,width=0.5\textwidth]{eps5e-2_xyc/inner_dynamics}}
\subfigure[\label{fig:conser_xyc}]{\includegraphics[angle=-90,width=0.5\textwidth]{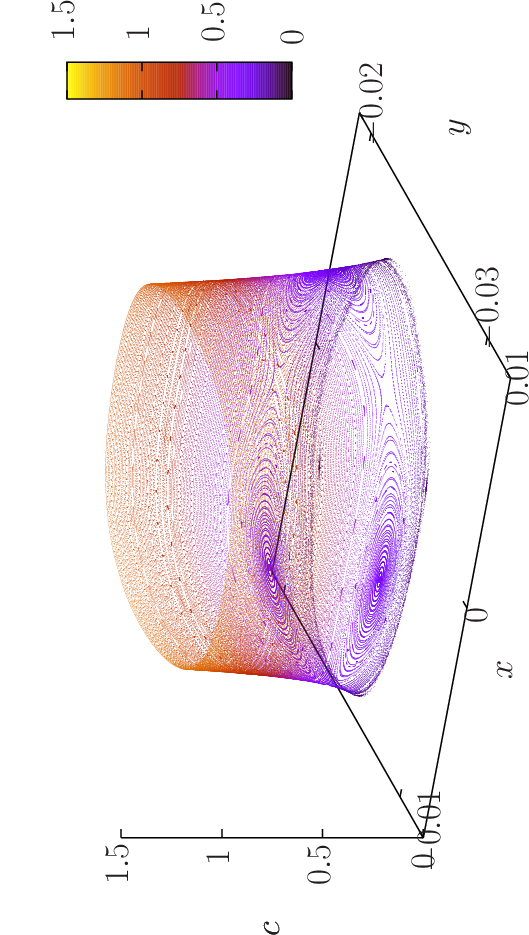}}
}
\end{picture}
\end{center}
\caption{Dynamics restricted to the Normally Hyperbolic Manifold
$\tL_\varepsilon$ in the ambient space for the conservative case.}
\label{fig:conser_ambient-space}
\end{figure}

In Figure~\ref{fig:conser_ambient-space} we show the inner dynamics in
the ambient space. In Figure~\ref{fig:conser_thetacw} we show the
variable $w$ parameterized by $\theta$ and $c$. Note that the
amplitude of the oscillations performed by $w$ increase with $c$.
These oscillations can be periodic or quasi-periodic depending on the
dynamics of $\theta$-$c$. In Figure~\ref{fig:conser_xyc} we see the
behaviour of $x$ and $y$. As sytems $\X$ and $\U$ are coupled through
the Hamiltonian coupling $\varepsilon h$ (the spring), the saddle
point $Q_0$ perturbs into an oscillatory motion.

As explained in Section~\ref{sec:parametrization_method}, the
Newton-like method reported in~\cite{HarCanFigLuqMon16} also provides
corrected versions of the normal bundle $N(\tL_\varepsilon)$ (the
matrix $N_\varepsilon(\theta,c)$, see Section~\ref{sec:step2}); that
is, linear approximations of the parameterizations
$W^{s,+}_\varepsilon$ and $W^{u,+}_\varepsilon$ given in
Equations~\eqref{eq:param_Wseps}-\eqref{eq:param_Wueps}:
\begin{align}
W^{s,+}(\theta,c,s,\tau,r)&=\tK_\varepsilon(\theta,c,s)+
(\tau,r)\cdot{\tilde{\Lambda}_\varepsilon^S(\theta,c)}
+O(\tau^2,r^2,\tau r)\label{eq:linear_stable}\\
W^{u,+}(\theta,c,s,\tau)&= \tK_\varepsilon(\theta,c,s)+
\tau\cdot {\tilde{\Lambda}^U_\varepsilon(\theta,c)}+O(\tau^2),\label{eq:linear_unstable}
\end{align}
where the matrices $\tilde{\Lambda}_\varepsilon^S$ and
$\tilde{\Lambda}_\varepsilon^U$ (with dimensions $2\times6$ and
$1\times6$) are the matrices $\Lambda_\varepsilon^S$ and
$\Lambda_\varepsilon^U$ (having dimensions $5\times 2$ and $5\times
1$) transformed into coordinates $x,y,u,v,w,s$ and properly
transposed.\\
When considering iterates of the linear approximations of the fibers
of the points $\tK_\varepsilon(\theta,c,s)$ one obtains better
approximations of the fibers of the corresponding inner iterates:
\begin{align}
\W^{s,+}(f^{-n}(\theta,c),s)&\simeq\tK_\varepsilon(f^{-n}(\theta,c),s)+
\stro_\varepsilon^{-n}\left( (\tau,r)\cdot
{\tilde{\Lambda}_\varepsilon^S(\theta,c)}\right)
\label{eq:iterated_stable}\\
\W^{u,+}(f^n(\theta,c),s)&\simeq
\tK_\varepsilon(f^n(\theta,c),s)+\stro^n_\varepsilon\left(\tau\cdot
{\tilde{\Lambda}^U_\varepsilon(\theta,c)}\right).\label{eq:iterated_unstable}
\end{align}
The higher $n$ and the smaller $\tau$ and $r$ are, the better the
approximation is.

\begin{figure}
\begin{center}
\begin{picture}(1,0.3)
\put(0,0.0){
\subfigure[]{\includegraphics[width=0.5\textwidth]{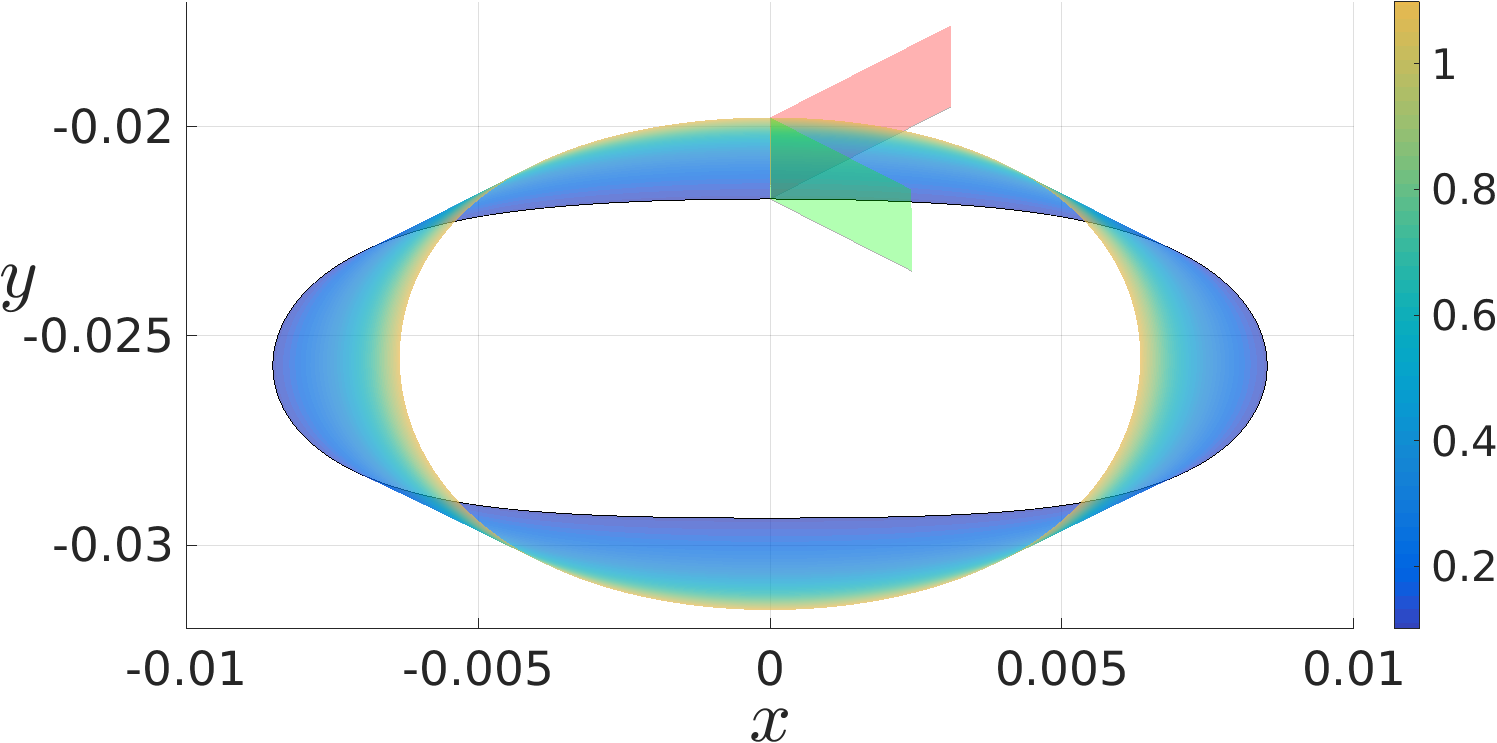}}
}
\put(0.51,0){
\subfigure[]{\includegraphics[width=0.5\textwidth]{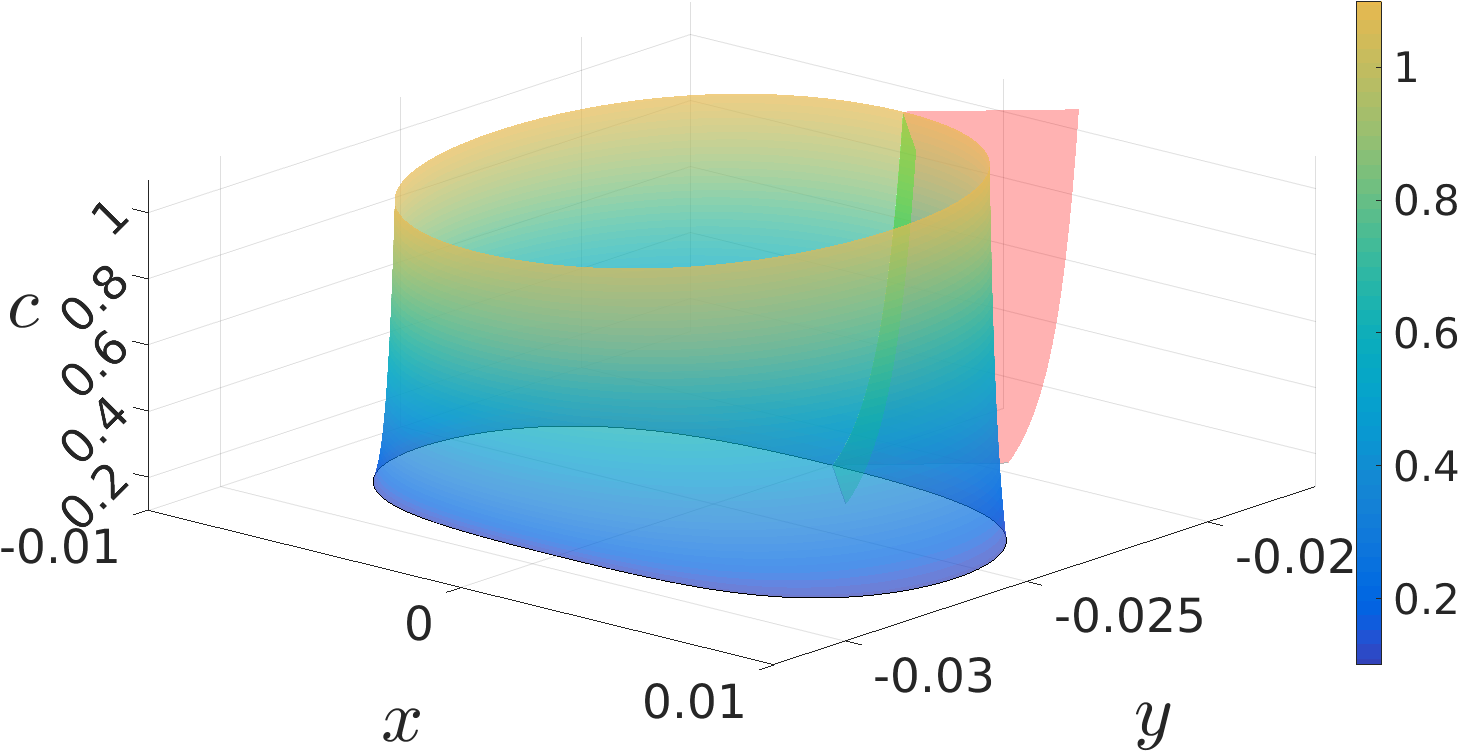}}
}
\end{picture}
\end{center}
\caption{Normally Hyperbolic Manifold and its normal bundle
$N(\tL_\varepsilon)$ for the conservative case: tangent space to the
stable (green) and unstable (red) fiber for
$\varepsilon=6\cdot10^{-2}$, $\tk=1$, $\tzeta=\tc=0$, $\theta=0.5$,
$c\in[0.2,1.2]$, $r=0$ and $\tau\in[0,5\cdot10^{-3}]$.}
\label{fig:conser_stable_unstable}
\end{figure}

To illustrate this, we show in Figure~\ref{fig:conser_stable_unstable}
the $x-y-c$ projection of the normal bundle of a set of points given
by $K_\varepsilon(\theta,c,s)$ with $\theta=0.5$, $c\in[0.1,1.2]$
$s=0$. For each such point we keep $r=0$ and slightly vary $\tau$,
$\tau\in[0,5\cdot 10^{-3}]$, in
Equations~\eqref{eq:linear_stable}-\eqref{eq:linear_unstable}.\\
\begin{figure}
\begin{center}
\includegraphics[width=0.8\textwidth]{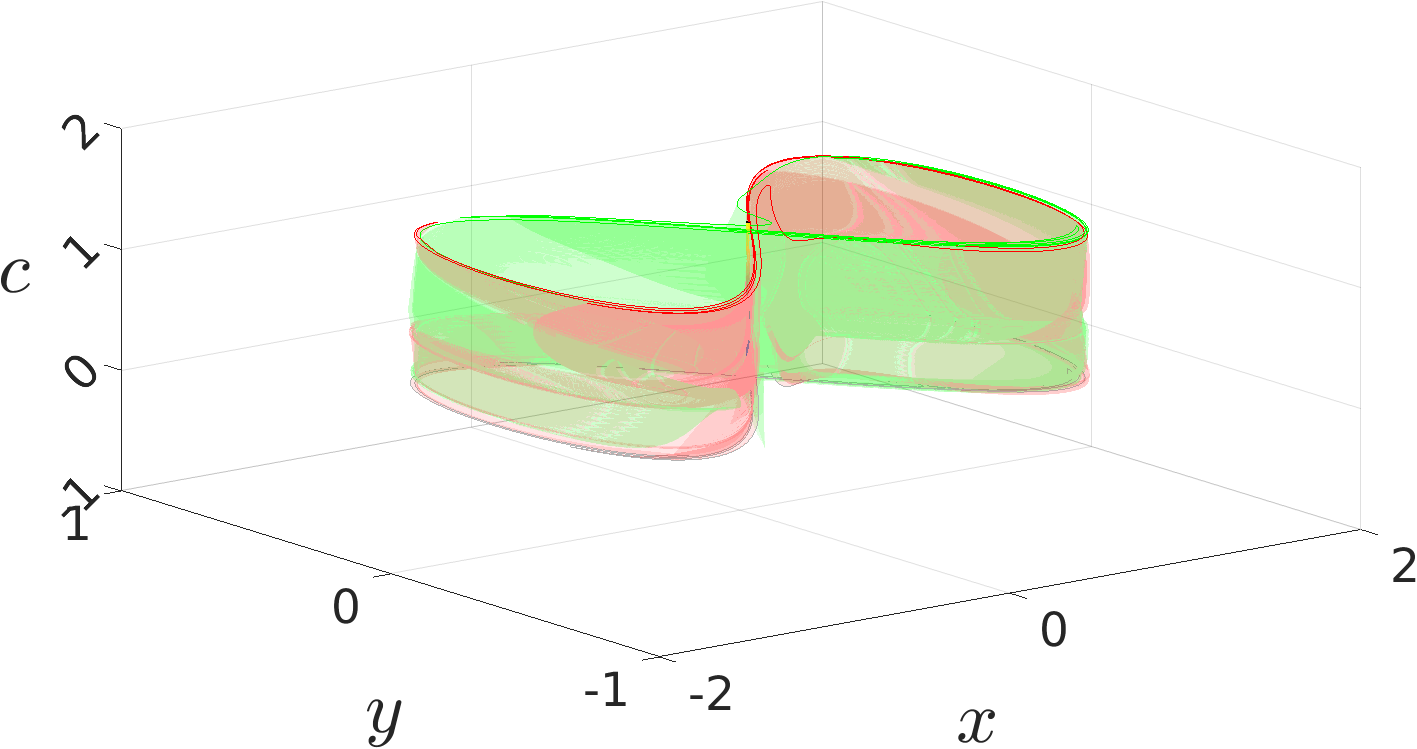}
\end{center}
\caption{Iteration of the normal bundle shown in
Figure~\ref{fig:conser_stable_unstable}. Green: $7$-th backwards
iteration of the stable bundle. Red: $7$-th forwards iteration of the
unstable bundle.}
\label{fig:conser_iterated_bundlexyc}
\end{figure}
\begin{figure}
\begin{center}
\includegraphics[width=0.8\textwidth]{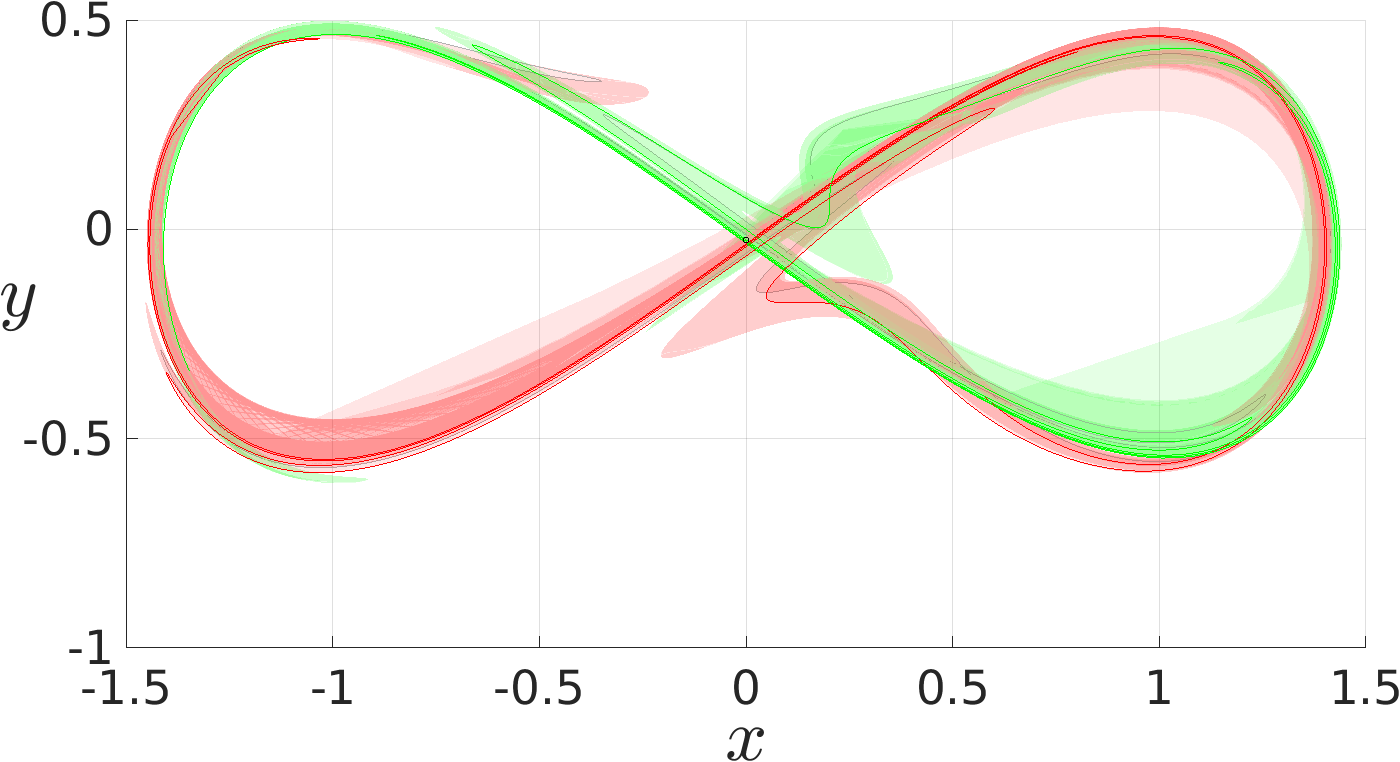}
\end{center}
\caption{Projection in the $x-y$ plane of
Figure~\ref{fig:conser_iterated_bundlexyc}.}
\label{fig:conser_iterated_bundlexy}
\end{figure}%
In Figures~\ref{fig:conser_iterated_bundlexyc}
and~\ref{fig:conser_iterated_bundlexy} we show the global
approximation of stable and unstable fibers by iterating $7$ times
($n=7$ in
Equations~\eqref{eq:iterated_stable}-\eqref{eq:iterated_unstable}) the
surfaces shown in Figure~\ref{fig:conser_stable_unstable}.\\
Figure~\ref{fig:conser_iterated_bundlexyc} shows evidence of
homoclinic intersections. Provided that the two beams are coupled by
means of a conservative coupling (a spring), there is in this case hope
to observe Arnold diffusion leading to $O(1)$ variations of the
coordinate $c$. A study of homoclinic intersections, the Scattering
map (see~\cite{DelLlaSea06,DelLlaSea08}) and shadowing trajectories is
left for future work.

Note that if the elastic constant of the spring is set to $\tk=0$,
then the two beams remain uncoupled and the energy, $c$, of system
$\U$ can only vary through the inner dynamics. Hence, in such a
situation, homoclinic excursions do not inject extra energy to the
beam represented by system $\U$. In other words, in this case, the
Scattering map becomes the identity up to first order terms and there
is no hope to observe Arnold diffusion in without the spring.

\subsection{Dissipative case}\label{sec:numerical_dissipative}
\subsubsection{Weak damping and conservative coupling}\label{sec:cons-coupling_damping}
When adding small dissipation, hyperbolicity of the saddle periodic
orbits guarantees their persistence for small enough dissipation.
However, as shown in~\cite{SimVie10}, when perturbed with dissipation,
elliptic periodic orbits of area preserving maps become attracting
foci.
\begin{figure}
\begin{center}
\includegraphics[angle=-90,width=0.8\textwidth]{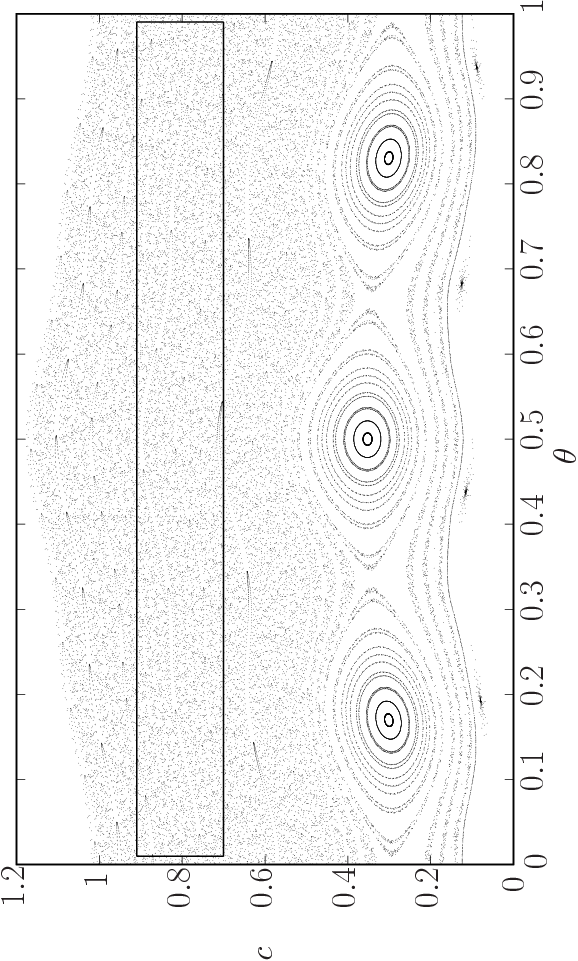}
\end{center}
\caption{Inner dynamics in $\tL_\varepsilon$ for weak damping:
$\varepsilon=6\cdot 10^-2$, $\tzeta=6\cdot 10^{-5}$ and $\tc=0$. We
have taken $50$ initial conditions at $\theta=0.5$ and iterated them
$1000$ times. The region labeled is maximized in
Figure~\ref{fig:7-3_onlydamping}, where $400000$ iterations are used.}
\label{fig:only-damping_global}
\end{figure}
As shown in Figure~\ref{fig:only-damping_global}, this occurs with
$f_\varepsilon$ as well. There, $50$ initial conditions are taken at
$\theta=0.5$ and iterated only $1000$ times for $\varepsilon=6\cdot
10^{-2}$, $\tc=0$, $\tk=1$ and $\tzeta=6\cdot 10^{-5}$, which
corresponds to an absolute magnitude of the damping coefficient of
$\varepsilon\tzeta=3.6\cdot10^{-6}$. Some of them are attracted to the
$3:1$ resonant focus, while others skip the separatrices of the saddle
$3:1$ periodic orbit and are attracted to lower energy attractors.
\begin{figure}%
\begin{center}
\includegraphics[angle=-90,width=0.8\textwidth]{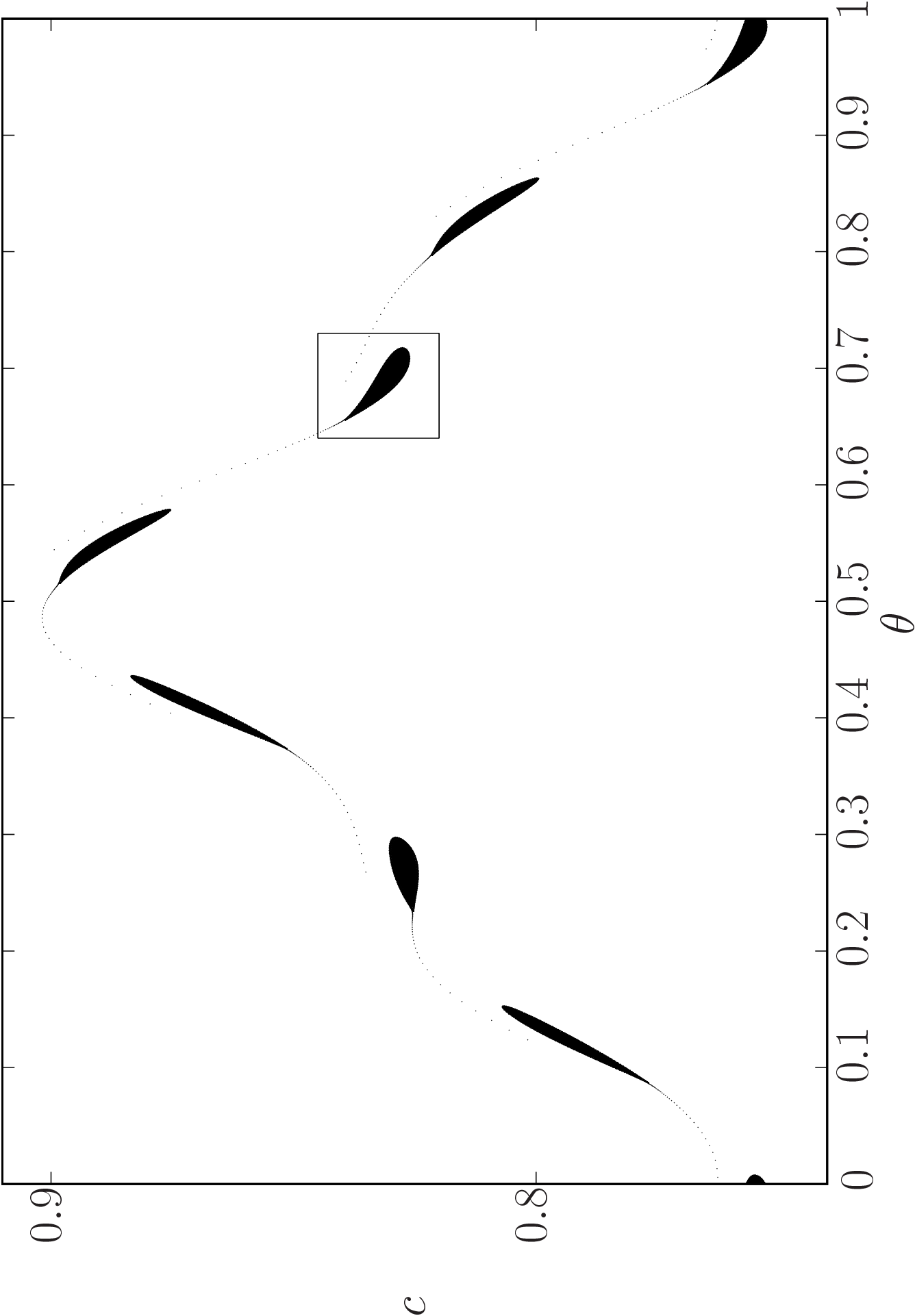}
\end{center}
\caption{Approximated unstable manifold of the saddle periodic orbits
corresponding to $7:3$ resonance under the presence of small damping:
$\tc=0$, $\tzeta=6\cdot 10^{-5}$, $\tk=1$ and $\varepsilon=6\cdot
10^{-2}$. The unstable manifold leaves the saddle point and rolls
about the $7:3$ resonant attracting focus. The labeled region is
magnified in Figure~\ref{fig:7-3_onlydamping_zoom}.}
\label{fig:7-3_onlydamping}
\end{figure}%
\begin{figure}
\begin{center}
\includegraphics[angle=-90,width=0.8\textwidth]{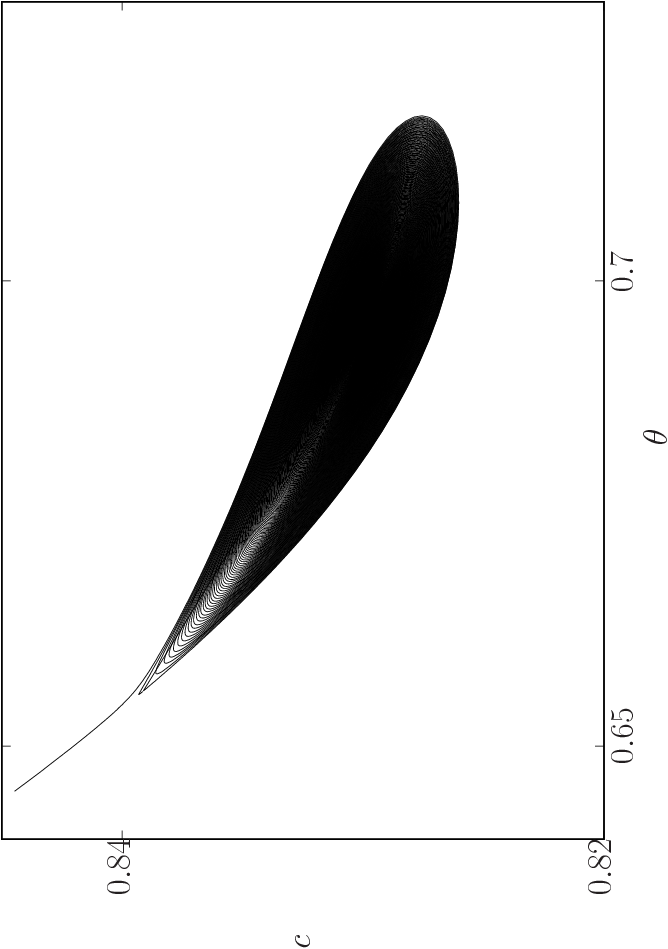}
\end{center}
\caption{Blow up of the labeled region in
Figure~\ref{fig:7-3_onlydamping}: unstable manifold of the $7:3$
resonant periodic orbit under weak damping.}
\label{fig:7-3_onlydamping_zoom}
\end{figure}%
In Figures~\ref{fig:7-3_onlydamping}
and~\ref{fig:7-3_onlydamping_zoom} we show this in more detail for the
$7:3$ resonant periodic orbits. There we have taken an initial
condition very close to the unstable manifold of the saddle $7:3$
resonant periodic orbit. For forward iterates we see how this unstable
manifold slowly rolls about the attracting focus while backwards
iterates approach the saddle periodic orbit and rapidly scape due to
limited numerical accuracy.  
\begin{figure}
\begin{center}
\begin{picture}(1,0.5)
\put(0,0.32){
\subfigure[\label{fig:7-3_iterates_theta}]{
\includegraphics[angle=-90,width=0.5\textwidth]{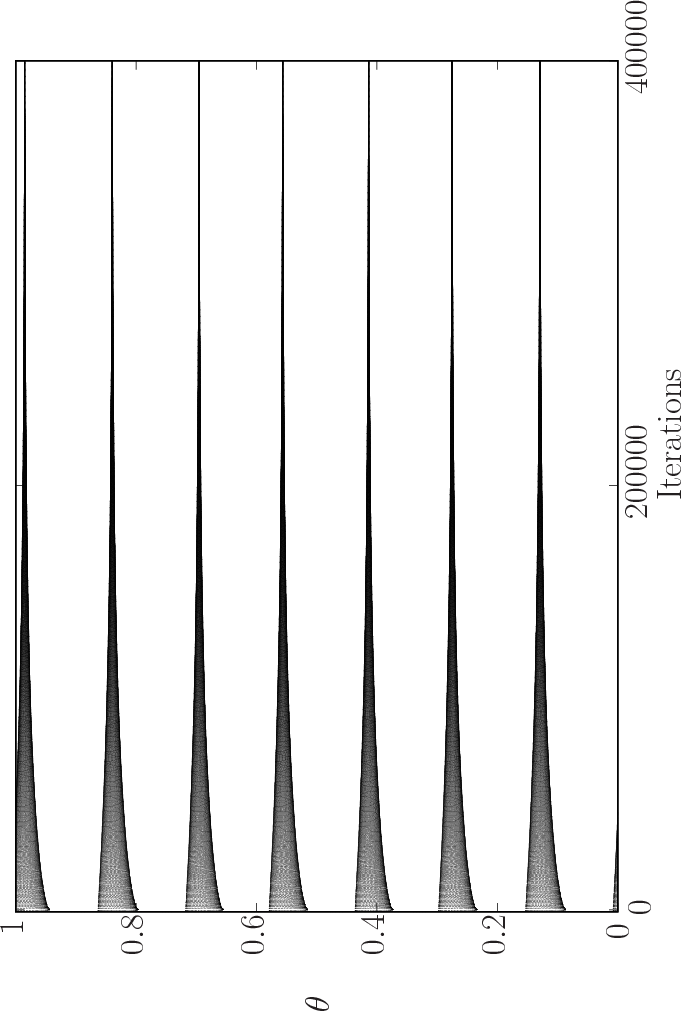}}
}
\put(0.51,0.32){
\subfigure[\label{fig:7-3_iterates_c}]{\includegraphics[angle=-90,width=0.5\textwidth]{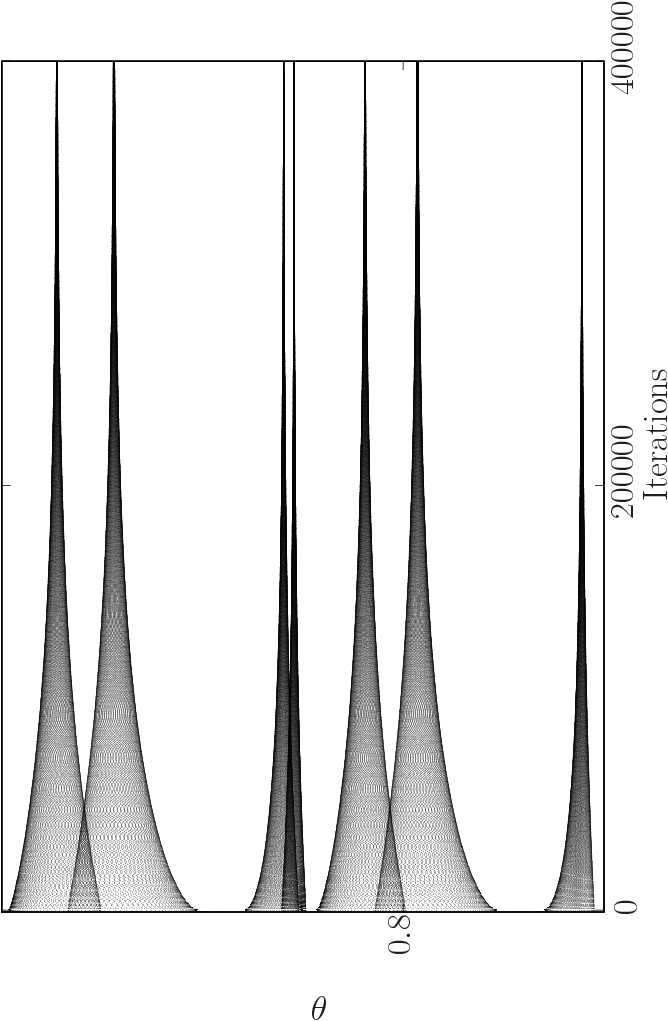}}
}
\end{picture}
\end{center}
\caption{Iterates of
$f_\varepsilon$ for the initial condition taken at the unstable
manifold shown in Figure~\ref{fig:7-3_onlydamping}
and~\ref{fig:7-3_onlydamping_zoom}.}
\label{fig:7-3_iterates}
\end{figure}%
\begin{figure}
\begin{center}
\begin{picture}(1,0.5)
\put(0,0.32){
\subfigure[\label{fig:7-3_iterates_theta_zoom}]{
\includegraphics[angle=-90,width=0.5\textwidth]{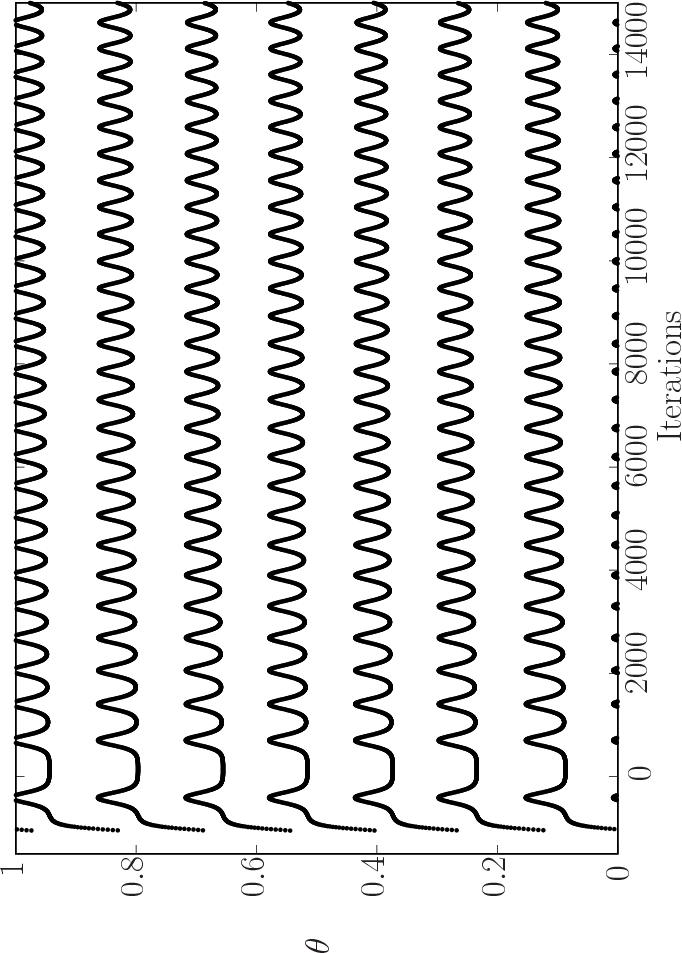}}
}
\put(0.51,0.32){
\subfigure[\label{fig:7-3_iterates_c_zoom}]{\includegraphics[angle=-90,width=0.5\textwidth]{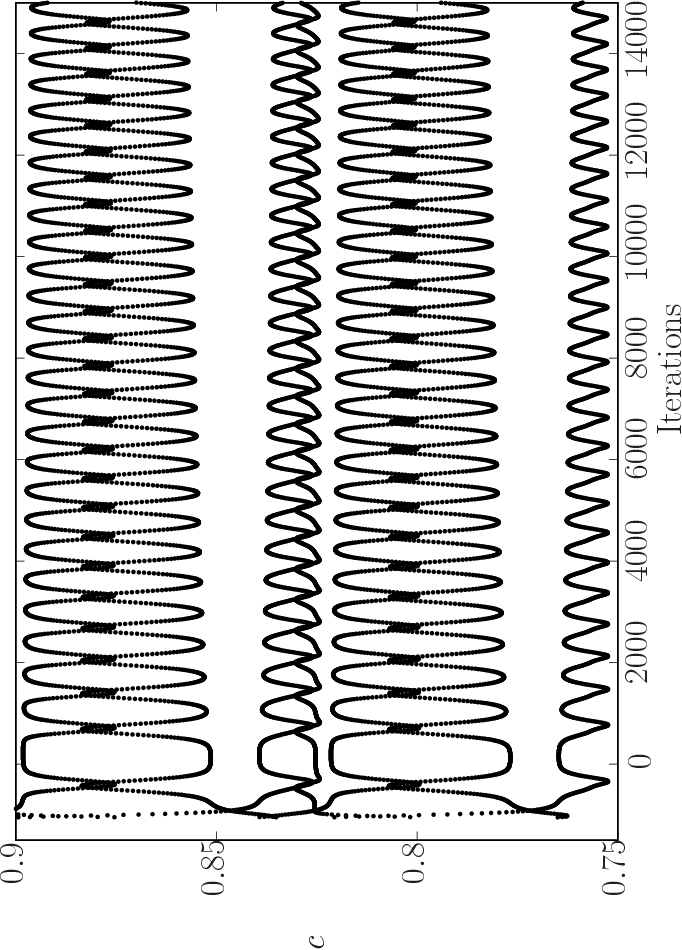}}
}
\end{picture}
\end{center}
\caption{First iterates of $f_\varepsilon$ and $f_\varepsilon^{-1}$
for the initial condition taken at the unstable manifold shown in
Figures~\ref{fig:7-3_onlydamping} and~\ref{fig:7-3_onlydamping_zoom}.}
\label{fig:7-3_iterates_zoom}
\end{figure}%
This is better appreciated in Figures~\ref{fig:7-3_iterates}
and~\ref{fig:7-3_iterates_zoom}, where we show the ``time'' evolution
of $\theta$ and $c$, both forwards and backwards in time.\\

The dissipation exhibited by the inner dynamics is indeed not desired
from the applied point of view, as it implies convergence to lower
energy oscillatory regimes providing lower amplitude alternate voltage
for variable $w$, which is the voltage provided to the load connected
to the harvesting beams (see Figure~\ref{fig:coupled_piezo_spring}).
As mentioned in the Introduction, one of the purposes of this work is
to provide tools that can prevent or slow down this loss of energy,
such as the ones based on outer excursions through homoclinic
intersections. We therefore are interested on studying the manifold
$\tL_\varepsilon$ and its stable and unstable manifolds for the
dissipative case.

Regarding the computations of these manifolds, we obtained results
very similar to those reported for the conservative case in
Section~\ref{sec:numerical_conservative}. That is, in the ambient
space, the manifold $\tL_\varepsilon$ and the normal bundle look very
similar to those shown in Figures~\ref{fig:conser_ambient-space}
and~\ref{fig:conser_stable_unstable}, respectively. Moreover, when
iterating the normal bundle, we obtain global fibers similar to the
ones shown in Figures~\ref{fig:conser_stable_unstable}
and~\ref{fig:conser_iterated_bundlexyc}. As the homoclinic
intersections shown for the conservative case in
Section~\ref{sec:numerical_conservative} are transversal, they are
robust to perturbations, even dissipative ones. Hence, we also observe
evidence of homoclinic connections allowing one to define the
Scattering map. Due to the (conservative) coupling, the Scattering map
may possess $O(\varepsilon)$ terms in the action $c$, although
expressions for this map for dissipative cases have not been reported
anywhere. Hence, through homoclinic excursions, one may inject
$O(\varepsilon)$ into the system which may help slowing down the
dissipation observed in the inner dynamics.  However, Arnold diffusion
is rather unlikely to exist due to the presence of dissipation.

%\clearpage

\subsubsection{Full system}\label{sec:full_system}
We now present the numerical results for the full case:
$\tk,\tzeta,\tc>0$.\\
In comparison with the previous case of
Section~\ref{sec:cons-coupling_damping}, we now add the extra
dissipative coupling term given by the coupling piezoelectric effect:
$\tc>0$.  Regarding the inner dynamics, computations reveal that, as
one would expect, the effect is similar to the situation given in
Section~\ref{sec:cons-coupling_damping} when the dissipation was only
due to the damping on the oscillators. We observe that the parameter
$\tc$ seems to contribute less than $\tzeta$ in destroying objects due to
dissipation.
\begin{figure}
\begin{center}
\includegraphics[angle=-90,width=0.8\textwidth]{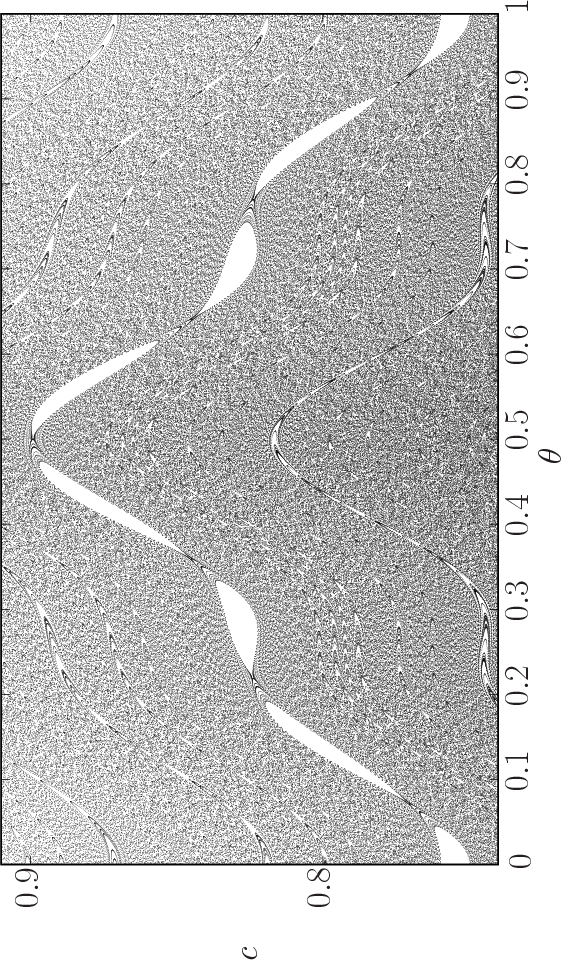}
\end{center}
\caption{Inner dynamics around the 7:3
resonant attracting focus for
$\varepsilon=6\cdot10^{-2}$, $\tzeta=6\cdot10^{-5}$ and $\tc=10^{-4}$.
}
\label{fig:7-3_onlycoupling_zoom}
\end{figure}%
In Figure~\ref{fig:7-3_onlycoupling_zoom} we show how the
$7:3$ resonant saddle and focus periodic orbits still persist for
$\tc=10^{-4}$.
\begin{figure}
\begin{center}
\includegraphics[angle=-90,width=0.8\textwidth]{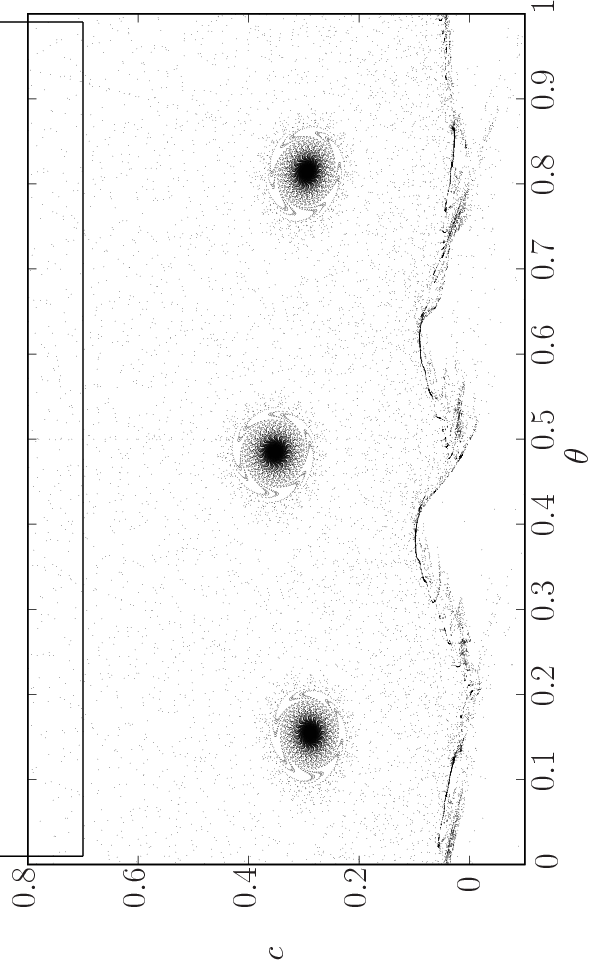}
\end{center}
\caption{Inner dynamics for $\varepsilon=6\cdot 10^{-2}$,
$\tzeta=6\cdot10^{-5}$ and $\tc=2\cdot10^{-1}$.}
\label{fig:only_coupling_global_tc2e-1}
\end{figure}%
As shown in Figure~\ref{fig:only_coupling_global_tc2e-1}, for larger
values of $\tc$, periodic orbits bifurcate and most initial conditions
are attracted towards a low energy attractor. For $\tc=2\cdot
10^{-1}$, the $3:1$ resonant periodic attracting focus still exists.
\begin{figure}
\begin{center}
\begin{picture}(1,0.5)
\put(0,0.32){
\subfigure[\label{fig:3-1_iterates_theta}]{
\includegraphics[angle=-90,width=0.5\textwidth]{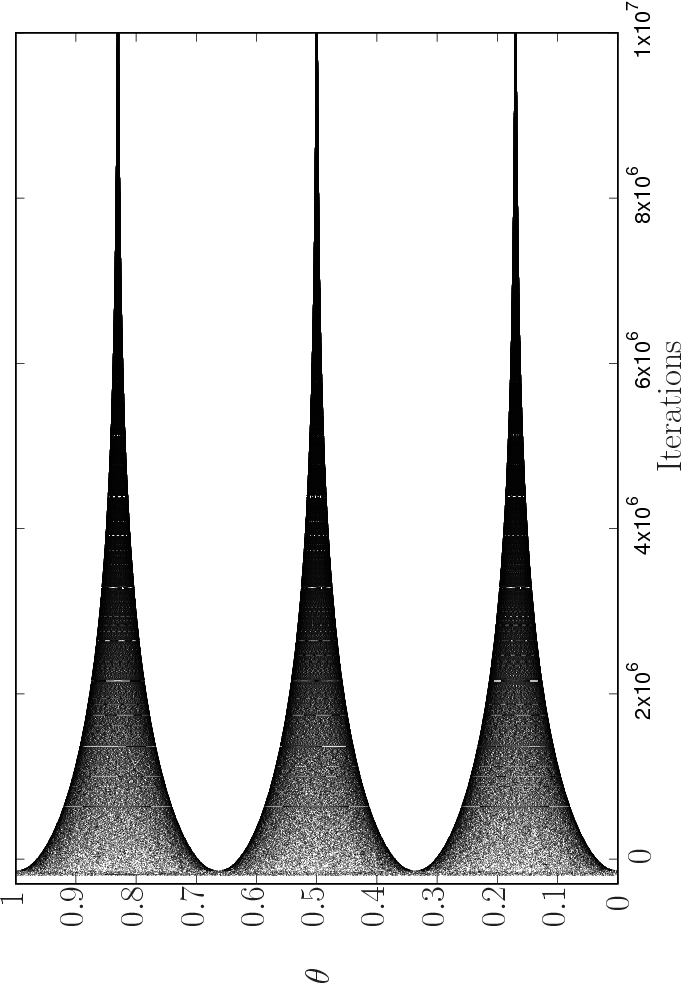}}
}
\put(0.51,0.32){
\subfigure[\label{fig:3-1_iterates_c}]{\includegraphics[angle=-90,width=0.5\textwidth]{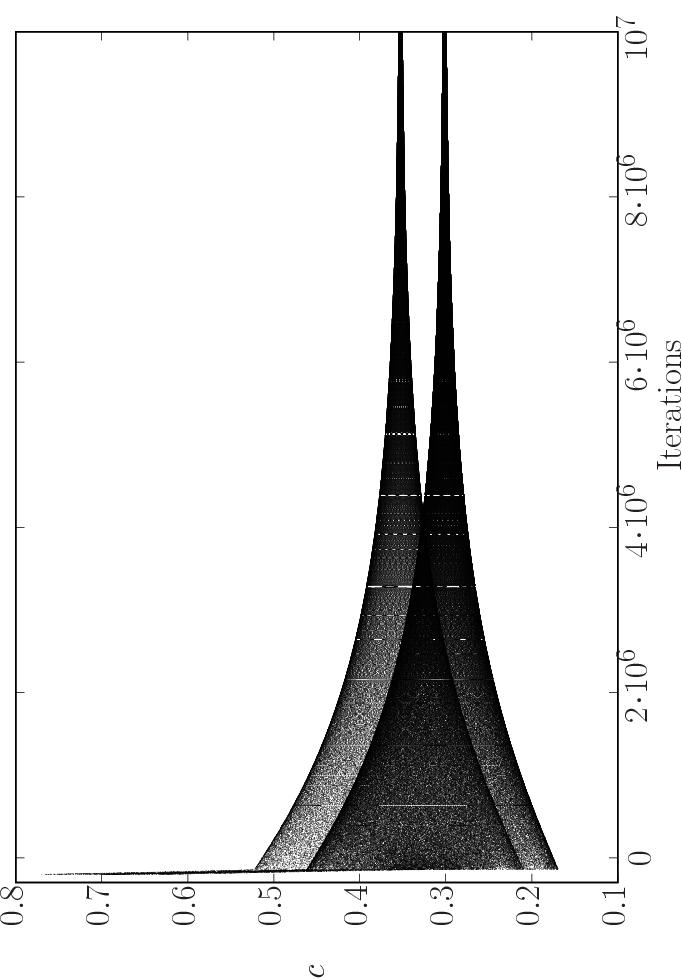}}
}
\end{picture}
\end{center}
\caption{Iterates by $f_\varepsilon$ and $f_\varepsilon^{-1}$
of an initial condition attracted by the $3$-periodic focus.}
\label{fig:3-1_iterates}
\end{figure}%
Iterates of $f_\varepsilon$ and $f^{-1}_\varepsilon$
close to the $3:1$ attracting focus are shown in
Figure~\ref{fig:3-1_iterates_theta}. Backwards iterates
are also shown until the trajectory scapes.

The manifold $\tK_\varepsilon$ in the ambient space and its normal
bundle $N(\tK_\varepsilon)$ look very  similar to the conservative
case considered in Section~\ref{sec:numerical_conservative}
(Figures~\ref{fig:conser_ambient-space}
and~\ref{fig:conser_stable_unstable}). We therefore omit including
including similar figures. However, in this case we also show the
$7th$ iterate of the normal bundle in
Figures~\ref{fig:full_system_iterated_bundlexyc1}
and~\ref{fig:full_system_iterated_bundlexy1}. As one can see there,
for the chosen parameter values,
there still exists evidence of existence of intersections between the
stable and unstable manifolds leading to homoclinic connections.
Hence, we show that there is hope that, through these homoclinic
connections, outer excursions can inject energy to the beam defined by
Hamiltonian $\U$ that may help the system slow down the loss of energy
shown by the inner dynamics previously discussed.

\begin{figure}
\begin{center}
\includegraphics[width=0.8\textwidth]{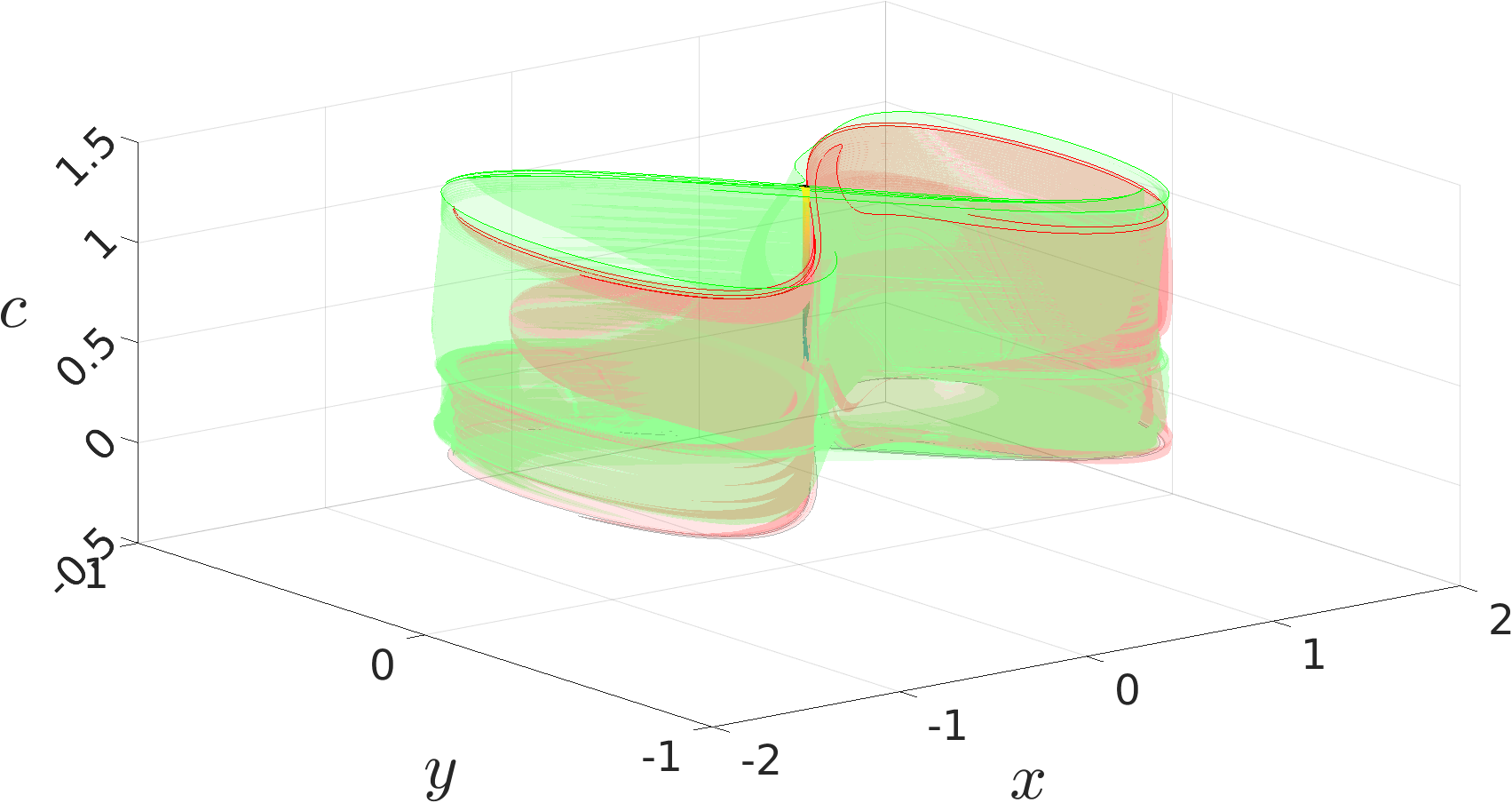}
\end{center}
\caption{Stable and unstable fibers of }
\label{fig:full_system_iterated_bundlexyc1}
\end{figure}
\begin{figure}
\begin{center}
\includegraphics[width=0.8\textwidth]{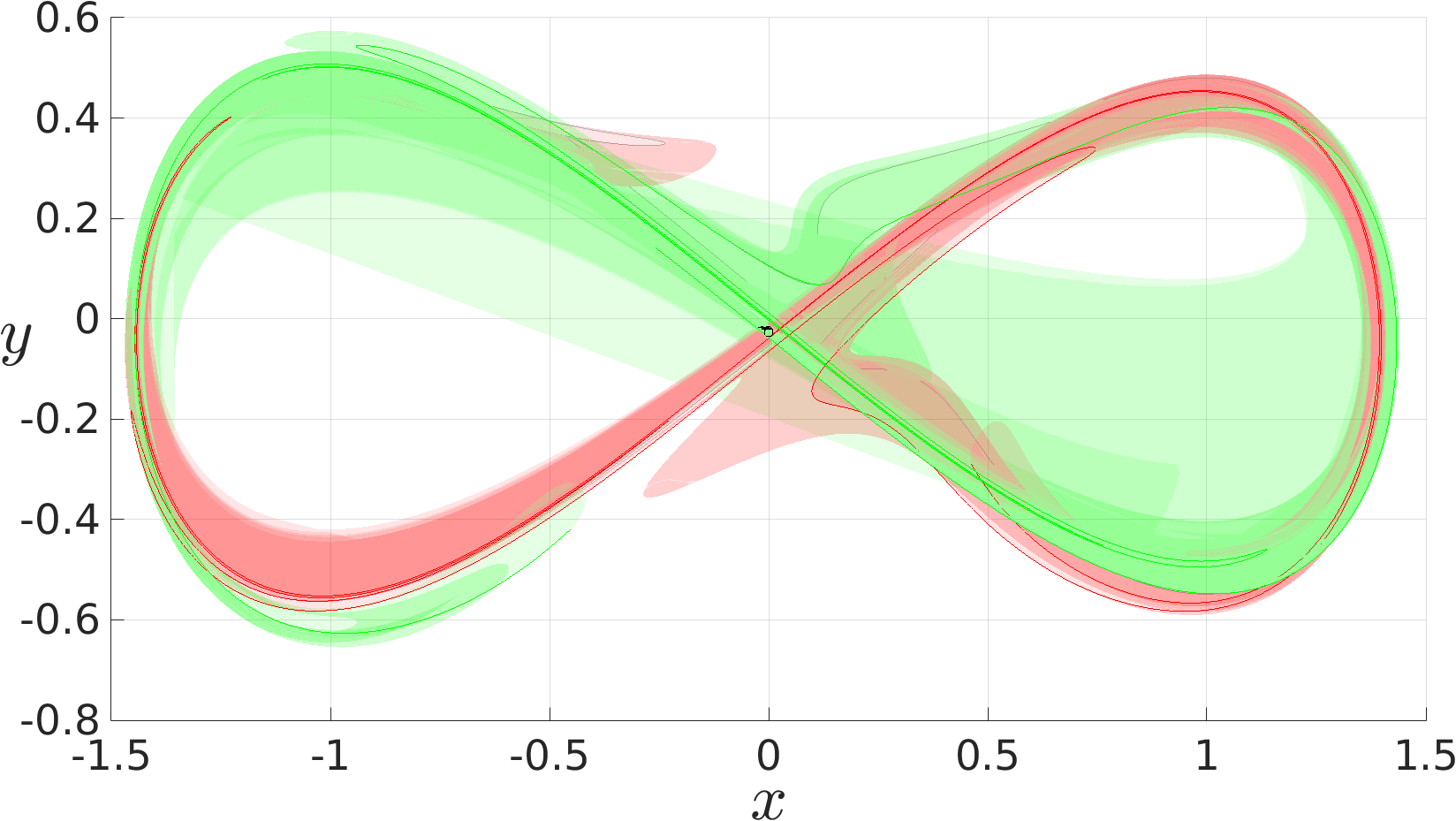}
\end{center}
\caption{Projection to the $x-y$ plane of
Figure~\ref{fig:full_system_iterated_bundlexyc1}}
\label{fig:full_system_iterated_bundlexy1}
\end{figure}

\section{Conclusions}\label{sec:conclusions}
This paper is a first step towards the use of theory related to Arnold
diffusion in energy harvesting systems based on bi-stable oscillators,
such as piezoelectric beams or cantilevers. Such theory could be
extremely useful in this field, as it precisely deals with the
accumulation of energy in oscillators absorbed from a given periodic
source.\\
The dynamics of such systems is given by the coupling of periodically
forced Duffing oscillators. The coupling is given by an extra variable
(a voltage) which at the same time adds extra dissipation to the
intrinsic damping. Moreover, this coupling adds and extra dimension to
the system. The goal of this work is to provide a theoretical and
numerical
background to study the existence and persistence of Normally
Hyperbolic Manifolds and the intersection between their unstable and
stable manifolds. Such intersections are the basis of the so-called
``outer dynamics'' in Arnold diffusion theory. Through these intersections,  the system may increase
its energy by absorbing energy from the source, which is studied by
the ``Scattering'' map. To benefit higher order of energy abortion,
we  have proposed to add to the system an extra conservative coupling
given by a spring. In the absence of damping and the piezoelectric
dissipative coupling, this extra coupling could allow the presence of
Arnold diffusion when periodically forced.

In the absence of forcing, damping and both couplings, we have proven
the existence of a $3$-dimensional Normally Hyperbolic Invariant
Manifold with $5$ and $6$-dimensional unstable and stable manifolds.
The unperturbed manifold possesses boundaries; despite the system's
dissipation, it persists and is unique.  However, in the presence of
dissipation, the inner dynamics becomes unbounded and hence the
manifold needs to be non-uniquely extended beyond the original
boundaries.

By implementing the Parameterization method we have computed this
manifold, its inner dynamics and good approximations of its stable and
unstable manifolds. We have numerically investigated three different
situations.\\
In the absence of damping and dissipative coupling, but including the
conservative one, the inner dynamics is given by
a symplectic map. The stable and unstable manifolds
intersect, giving rise to outer excursions through homoclinic
intersections. In this case, the Scattering map could be properly
defined, and first order terms could be computed as usual.\\
When the damping is enabled, the inner dynamics is not given by an
area preserving map anymore.  Instead, the inner dynamics at the
manifold possesses global attractors to which trajectories are
attracted losing energy. However, we have shown evidence of existence
of homoclinic connections.  Such intersections may lead to outer
excursions injecting energy, which could be used to overcome or slow
down the loss of energy given at inner dynamics. This is extremely
desired from the applied point of view and may help to optimize energy
harvesting systems based on this type of oscillators. However, the
system may not exhibit Arnold diffusion anymore due to the presence of
dissipation.\\
We have finally numerically studied the full system, and shown that a
similar situation applies up to higher values of the piezoelectric
coupling.  

We propose to continue our work by providing a theoretical background
for the existence of homoclinic intersections (Melnikov theory) and
the Scattering map for dissipative systems, on one hand. On the other
hand, we also propose an accurate numerical computation of homoclinic
intersections and the Scattering map, in order to quantify the amount
of absorbed energy from the source.
\bibliographystyle{plain}
\bibliography{DynamicalSystems.bib,harvesting.bib,melnikov_like.bib,we.bib}

\end{document}